\documentclass{amsart}[12]
\usepackage{amsmath}
\pdfoutput=1

\usepackage{graphicx}
\usepackage{latexsym} 
\usepackage{color}
\usepackage{amscd}
\usepackage[all]{xy}
\usepackage{enumerate}
\usepackage{hyperref}
\usepackage{subfigure}
\usepackage{soul}
\usepackage{comment}
\usepackage{epsfig}
\usepackage{epstopdf}
\usepackage{marginnote}
\usepackage{multirow}

\newtheorem{thm}{Theorem}
\newtheorem{add}[thm]{Addendum}

\newtheorem*{thma}{Theorem~A}
\newtheorem*{thmb}{Theorem~B}

\newtheorem{theorem}[thm]{Theorem}
\newtheorem{lemma}[thm]{Lemma}

\newtheorem{proposition}[thm]{Proposition}

\theoremstyle{definition}

\newtheorem*{definition*}{Definition}

\newtheorem{remark}[thm]{Remark}

\newcommand{\CPb}{\overline{\mathbb{CP}}{}^{2}}
\newcommand{\CP}{{\mathbb{CP}}{}^{2}}

\newcommand{\N}{\mathbb{N}}
\newcommand{\Z}{\mathbb{Z}}

\newcommand{\D}{\operatorname{Diff{^+}}}

\newcommand{\twprod}{\mathbin{\mathchoice%
    {\ooalign{\raise1.15ex\hbox{$\scriptstyle\sim$}\cr\hidewidth$\times$\hidewidth\cr}}%
    {\ooalign{\raise1.15ex\hbox{$\scriptstyle\sim$}\cr\hidewidth$\times$\hidewidth\cr}}%
    {\ooalign{\raise.85ex\hbox{$\scriptscriptstyle\sim$}\cr\hidewidth$\scriptstyle\times$\hidewidth\cr}}%
    {\ooalign{\raise.65ex\hbox{$\scriptscriptstyle\sim$}\cr\hidewidth$\scriptscriptstyle\times$\hidewidth\cr}}%
    }}

\newcommand{\M}{\operatorname{Mod}}
\newcommand{\Spin}{\operatorname{Spin}}

\usepackage{scalerel} 
\newcommand{\DT}[1]{t_{\scaleto{\mathstrut #1}{9pt}}}
\newcommand{\LDT}[1]{t_{\scaleto{\mathstrut #1}{9pt}}^{-1}}

\def \x {\times}
\def \eu{{\text{e}}}

\allowdisplaybreaks[4]

\begin{document}

\title[Exotic $4$--manifolds with signature zero  ] 
{Exotic $4$--manifolds with signature zero  }

\author[R. \.{I}. Baykur]{R. \.{I}nan\c{c} Baykur}
\address{Department of Mathematics and Statistics, University of Massachusetts, Amherst, MA 01003, USA}
\email{inanc.baykur@umass.edu}

\author[N. Hamada]{Noriyuki Hamada}
\address{Institute of Mathematics for Industry,
Kyushu University, Motooka 744, Nishi-ku, Fukuoka 819-0395,
Japan}
\email{hamada.noriyuki.051@m.kyushu-u.ac.jp}

\begin{abstract} 
We produce infinitely many distinct irreducible smooth $4$--manifolds homeomorphic to $\#_{2m+1}(\CP \,\#\, \CPb)$ and $\#_{2n+1} (S^2 \times S^2)$, respectively, for each  $m \geq 4$ and $n \geq 5$. These provide the smallest exotic closed simply connected \mbox{$4$--manifolds} with signature zero known to date, and in each one of these homeomorphism classes, we get minimal symplectic $4$--manifolds. Our novel exotic $4$--manifolds are derived from fairly special small Lefschetz fibrations we build via positive factorizations in the mapping class group, with spin and non-spin monodromies.
\end{abstract}

\maketitle

\setcounter{secnumdepth}{2}
\setcounter{section}{0}


\vspace{0.15in}
\section{Introduction} 

Despite the spectacular advances in the constructions of exotic smooth structures on closed $4$--manifolds, signature zero $4$--manifolds with small topology have been quite resilient to these epic efforts. Our main theorem addresses this outstanding problem.

\begin{thma}\
There exist infinitely many, pairwise non-diffeomorphic, irreducible \mbox{$4$--manifolds} homeomorphic to $\#_{2m+1} (\CP \# \CPb)$ and $\#_{2n+1} (S^2 \times S^2)$, respectively, for every  $m \geq 4$, $n \geq 5$. Each family contains minimal symplectic $4$--manifolds, as well as \mbox{$4$--manifolds} that do not admit a symplectic structure. All of them are derived by surgeries along a link of Lagrangian tori in  symplectic $4$--manifolds explicitly described as Lefschetz fibrations over a genus--$2$ surface. 
\end{thma}

Recall that $b_2^+(X)$ is odd for a closed simply connected symplectic $4$--manifold $X$. When, in addition, the signature of $X$ is zero, by Freedman's theorem \cite{Freedman}, $X$ falls into one of the two homeomorphism classes listed in the theorem: the connected sum of an odd number of copies of $\CP \# \CPb$ or of $S^2 \times S^2$, depending on whether $X$ has an odd or even intersection form. Together with the previous results in the literature \cite{Park1, ABBKP, AkhmedovParkOdd, Park2, ParkSzabo}, we, in particular, settle the existence of symplectic closed simply connected \mbox{$4$--manifolds} with Chern numbers $c_1^2
\leq 2c_2$, for all possible pairs but 
$(c_2, c_1^2)=(8,16),(12,24),(16,32)$ and $(20,40)$ in the spin case. 
Furthermore, for each $m 
\geq 5$ and $n \geq 6$, we produce infinitely many symplectic $4$--manifolds that are not diffeomorphic to each other or any complex surface, in each one of the homeomorphism classes of $\#_{2m+1} (\CP \# \CPb)$ and $\#_{2n+1} (S^2 \times S^2)$, respectively; see~Addendum~\ref{add:knotsurgery}. We get examples with even smaller topology as we allow non-trivial fundamental group; in Addendum~\ref{add:smallest}, we produce infinitely many exotic copies of $(T^2 \x S^2) \#_ {8} (\CP \, \# \, \CPb)$ and $(T^2 \x S^2) \#_{10} (S^2 \x S^2)$.

The first examples of irreducible $4$--manifolds in the homeomorphism classes listed in Theorem~A were given by Jongil Park in \cite{Park1, Park2} for large $m$ and $n$. In subsequent works by various authors over the past 20 years, these results were  improved to cover smaller and smaller $m$ and $n$, culminating in the best results in the non-spin case for $m \geq 8$, by Akhmedov, Sakall{\i} and Yeung in \cite{ASY}, and in the spin case for $n \geq 11$ by our earlier work in \cite{BaykurHamada}. Theorem~A  is a simultaneous improvement of these results. 

Moreover, all earlier constructions of exotic $\#_{2m+1} (\CP \# \CPb)$ and $\#_{2n+1} (S^2 \times S^2)$ utilize  complex surfaces with non-negative signatures built by algebraic geometers as essential ingredients, usually with only partial understanding of their smooth topology. Our construction of the exotic $4$--manifolds in Theorem~A is all explicit; for instance, one can algorithmically describe their Kirby diagrams; see Remark~\ref{rk:kirby}. Our main ingredients are symplectic Lefschetz fibrations over the genus--$2$ surface that we build explicitly via positive factorizations in mapping class groups, and they are not homotopy equivalent to any complex surface. As we derive our examples via surgeries along a link of Lagrangian tori in these fibrations, our constructions can be regarded as instances of the reverse engineering approach of Fintushel and Stern \cite{FPS}. Our success in killing the fundamental group via surgeries along a link of Lagrangian tori boils down to the fact that our models contain what we call \emph{kernels}, certain codimension zero symplectic submanifolds with embedded disks relative to their boundary; see Remark~\ref{rk:kernel} and Proposition~\ref{prop:Luttinger}. In our examples, these kernels appear as symplectic subfibrations over a punctured torus. In our search for the latter, we obtain the following results:

\begin{thmb}
There are infinitely many relatively minimal symplectic genus--$g$ Lefschetz fibrations 
$\{(X_{g,n}, f_{g,n})\}$ and $\{(Y_{g,n}, h_{g,n})\}$ over the $2$--torus with signature zero,  which have singular sets consisting of only four nodes, and with non-spin and spin monodromies, defined for all $n \in \N$, and for $g \geq 4$ and $g \geq 5$, respectively.  The total spaces $X_{g,n}$ and $Y_{g,n}$ are irreducible non-spin and spin symplectic $4$--manifolds. In each infinite family (for fixed $g$), the total spaces are not homotopy equivalent to each other or to any complex surface. 
\end{thmb}

We build the spin fibrations in the theorem by producing positive factorizations, which involve positive Dehn twists and commutators, in \emph{spin mapping class groups}. In Proposition~\ref{prop:spin}, we explain how to equip the total spaces of these fibrations over arbitrary surfaces with spin structures. Working in this smaller mapping class group is crucial for our construction of exotic \mbox{$\#_{2n+1} (S^2 \times S^2)$.} 

\vspace{0.1in}
\noindent \textit{Acknowledgements.} The first author would like to thank the Harvard University, Max Planck Institute for Mathematics in Bonn and  Erd\"{o}s Center in Budapest for their hospitality while writing this article. This work was supported by the NSF grant DMS-2005327. 
The second author was supported by JSPS KAKENHI Grant Number 20H01802.

\smallskip
\section{Preliminary results}

Here we summarize our conventions, review the main definitions and background results needed throughout the paper, and prove a few preliminary results for later sections. 
We refer the reader to ~\cite{GompfStipsicz}, ~\cite{FarbMargalit} and ~\cite{BaykurHamada} for more details on Lefschetz fibrations, symplectic $4$--manifolds, mapping class groups and spin monodromies.

\noindent \textit{\underline{Conventions}:} All the manifolds and the maps between them we consider in this article are smooth. We denote a compact orientable surface of genus $g$ with $b$ boundary components by $\Sigma_{g}^b$,  whereas we drop $b$ from the notation when there is no boundary. We denote by $\D(\Sigma_g^b)$ the group of orientation--preserving diffeomorphisms 
of $\Sigma_g^b$ that restrict to the identity around the boundary. The \textit{mapping class group} of $\Sigma_g^b$ is $\M(\Sigma_{g}^b):=\pi_0(\mathrm{Diff}^+(\Sigma_g^b))$. By curves on $\Sigma_g^b$, we always mean simple closed curves. We consider them and the elements in $\D(\Sigma_g^b)$ up to isotopy and denote their isotopy classes by the same symbols. The products of mapping classes act on curves starting with the rightmost factor. The right-handed, or the \textit{positive Dehn twist}, along a curve $c$ on $\Sigma_g^b$ is denoted by $t_c$ in $\M(\Sigma_g^b)$. For any $A$ and $B$ in $\M(\Sigma_g^b)$, we denote the commutator $[A,B]:=A B A^{-1} B^{-1}$ and the conjugate $A^{B}:=BA B^{-1}$. The index set $\N$ we often use is the set of non-negative integers. Whenever no coefficient group is specified, $H_*(\_)$ and $H^*(\_)$ denote the the integral (co)homology groups.

\subsection{Lefschetz fibrations and monodromy factorizations} \

Let $X$ and $\Sigma$ be compact connected oriented smooth manifolds of dimensions $4$ and $2$. A map $f \colon X \to \Sigma$ is a \textit{Lefschetz fibration} over $\Sigma$ if around every critical point $f$ there are orientation-preserving complex coordinates in which $f$ is given by $(z_1,z_2) \mapsto z_1^2+z_2^2$.
We assume no fiber of $f$ contains an embedded sphere of self-intersection $-1$ (i.e., the fibration is \emph{relatively minimal}) or more than one critical point.  Any regular fiber $F$ of $f$ is diffeomorphic to a fixed $\Sigma_g^b$, where $g$ is called the \textit{genus} of the fibration. 

Let $X$ be a closed  oriented $4$--manifold.  A \textit{monodromy factorization} or \textit{positive factorization} for a genus--$g$ Lefschetz fibration $(X,f)$ over a genus--$h$ surface, with $b$ disjoint sections $\{S_i\}$ of self-intersections $S_i \cdot S_i=-k_i$ is a relation of the form 
\begin{equation} \label{eq:monodromy}
 t_{c_1} t_{c_2} \, \cdots t_{c_\ell} [A_1, B_1] \cdots [A_h, B_h] =  t_{\delta_1}^{k_1} \cdots t_{\delta_b}^{k_b}  
   \ \ \ \text{ in }\M(\Sigma_g^b),
\end{equation}
where $\{\delta_i\}$ are boundary parallel curves along distinct boundary components of $\Sigma_g^b$.\footnote{To provide the additional information on the sections, here the monodromy appears to be described only in the complement of the sections, but it naturally extends to all of $X$ as a factorization of identity in $\M(\Sigma_g)$ under the homomorphism induced by the inclusion $\Sigma_g^b \hookrightarrow \Sigma_g$.} 
Here, each $t_{c_i}$ comes from the local monodromy around a critical value, with respect to a reference regular fiber $F$ identified with $\Sigma_g \supset \Sigma_g^b$, where the singular nodal fiber is obtained by crushing the \emph{vanishing cycle} $c_i$ on $F$, whereas the $A_i, B_i$ are the images of a standard basis of $\pi_1(\Sigma_h^1)$ under the monodromy homomorphism for the $\Sigma_g^b$--bundle obtained by restricting $f$ over the complement of a disk $D \subset \Sigma_h$ that contains all the  critical values. Note that when $\ell=0$, the relation~\eqref{eq:monodromy} is the monodromy factorization of a $\Sigma_g$--bundle over $\Sigma_h$ with sections.

Conversely, any relation of the form~\eqref{eq:monodromy} prescribes  a Lefschetz fibration $(X,f)$ as above with a specific handle decomposition, and the basic algebraic topology of $X$ can be easily inferred from the monodromy factorization.  (When $\ell=0$ and $g \leq 1$, one needs additional information, but we do not deal with these cases.) The Euler characteristic is given by the formula $\eu(X)=4(g-1)(h-1)+ \ell$ and the signature $\sigma(X)$ can be calculated  algorithmically using \cite{Endo, EndoNagami} if enough is known about how the relation~\eqref{eq:monodromy} is derived in the mapping class group, or by invoking  \cite{Ozbagci}. With Poincar\'{e} duality in hand, to determine all of $H_*(X)$ and $H^*(X)$ it remains to pin down $H_1(X)$, which can be calculated by abelianizing the presentation for $\pi_1(X)$ we describe next. 

Assume that $(X,f)$ has a section, that is, $b>0$ in~\eqref{eq:monodromy}. Let $\mathcal{B}_F:=\{a_1, b_1, \cdots, a_g, b_g \}$ and $\mathcal{B}_B:=\{ x_1, y_1, \cdots, x_h, y_h \}$ be standard generating sets for $\pi_1(\Sigma_g)$ and $\pi_1(\Sigma_h)$, respectively, so that the surface relations are of the forms $\Pi_{i=1}^g[a_i,b_i]=1$ and $\Pi_{j=1}^h[x_j,y_j]=1$. For $D \subset \Sigma_h$ a disk containing all the critical values, $f$ is a union of a Lefschetz fibration over $D$ and a $\Sigma_g$--bundle over $\Sigma_h^1 \cong \Sigma_h \setminus D$. 
By examining the handle decomposition associated with the fibration over $D$, as well as the short exact sequence of homotopy groups associated with the bundle over $\Sigma_h^1$, one concludes the following:

\begin{proposition} \label{prop:pi1}
There is a presentation for $\pi_1(X)$ with generators  $\mathcal{B}_F \cup \mathcal{B}_B$ and relations
\begin{enumerate}[(i)]
	\item
	$\prod_{i=1}^g[a_i,b_i]=1$ and $\prod_{j=1}^h[x_j,y_j]=1$,
	\item
	$c_i=1$, for $i=1, \ldots, \ell$, 
	\item
	$x_jcx_j^{-1} = (A_j)_*(c)$ and $y_jcy_j^{-1}=(B_j)_*(c)$ for any $c \in \mathcal{B}_F$ and $j=1, \cdots, h$.
\end{enumerate}
\end{proposition}

\noindent Here each vanishing cycle $c_i \subset \Sigma_g$, taken with an auxiliary orientation and a base path, is represented in the basis $\mathcal{B}_F$, whereas $(A_j)_*$ and $(B_j)_*$ are the $\pi_1(\Sigma_g)$ automorphisms induced by any representatives of $A_j$ and $B_j$ in $\D(\Sigma_g)$.

Lastly, we record a few key relations \cite{FarbMargalit, KorkmazOzbagci} in $\M(\Sigma_g^b)$ we are going to use repeatedly:\footnote{In fact, together with the elementary relators $t_\alpha t_\alpha^{-1}=1$ for any $\alpha$, and $t_\alpha=1$ for any null-homotopic curve $\alpha$, the three types of relations we list here generate all other relations among Dehn twists in $\M(\Sigma_g^b)$.}

\noindent \textit{Braid relation}:  
		$$t_{\alpha} t_{\beta} = t_{t_{\alpha} (\beta)} t_{\alpha} $$
for any simple closed curves $\alpha$, $\beta$ on $\Sigma_g^b$. In particular, the braid relation implies that $t_{\alpha} t_{\beta} = t_{\beta} t_{\alpha} $ when $\alpha$ and $\beta$ are disjoint and that $t_{\alpha} t_{\beta} t_{\alpha} = t_{\beta} t_{\alpha} t_{\beta} $ when $\alpha$ and $\beta$ intersect transversely at one point.
We will freely use braid relations in our calculations without mentioning. 

\smallskip
\noindent \textit{Lantern relation}: 
		$$t_{\alpha} t_{\beta} t_{\gamma}= t_{\delta_1} t_{\delta_2} t_{\delta_3} t_{\delta_4} $$
		for the curves in Figure~\ref{F:LanternRelation}.

\smallskip
\noindent \textit{Four-holed torus relation}:  
		$$(t_{\alpha_1} t_{\alpha_3} t_{\beta} t_{\alpha_2} t_{\alpha_4} t_{\beta})^2 = t_{\delta_1} t_{\delta_2} t_{\delta_3} t_{\delta_4} $$ 
		for the curves in Figure~\ref{F:4holedTorusRelation}.

\begin{figure}[htbp]
	\centering
	\subfigure[$\Sigma_0^4$ with boundary $\{ \delta_1, \delta_2, \delta_3, \delta_4 \}$.\label{F:LanternRelation}]
	{\hspace{25pt}
		\includegraphics[height=90pt, trim=0 -10 0 10]{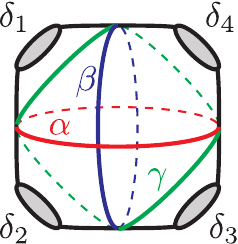} 
		\hspace{25pt}} 
	\hspace{1em}	
	\subfigure[$\Sigma_1^4$ with boundary $\{ \delta_1, \delta_2, \delta_3, \delta_4 \}$.\label{F:4holedTorusRelation}]
	{\hspace{10pt}
		\includegraphics[height=110pt]{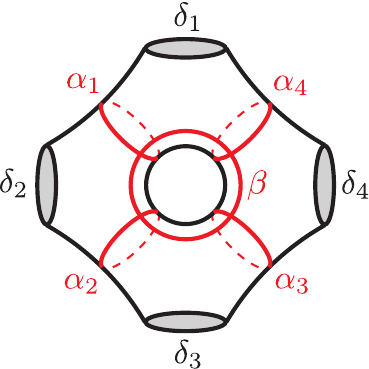}
		\hspace{10pt}} 
	\caption{Curve configurations for the lantern and the four-holed torus relations.  Each one is viewed on a subsurface of $\Sigma_g^b$.} \label{F:KeyRelations}
 \end{figure}

\smallskip
\subsection{Symplectic $4$--manifolds and surgeries along tori} \

By the Gompf-Thurston construction, any Lefschetz fibration $(X,f)$ with a non-torsion fiber class $[F]$ in $H_2(X)$ (e.g. when the fiber genus $g \neq 1$ or there is a section) admits a symplectic structure $\omega$ with respect to which the fibers and any preassigned collection of disjoint sections are symplectic; moreover, any such $\omega_1$ and $\omega_2$ are deformation equivalent. Throughout the paper we equip our Lefschetz fibrations (with sections) always with a Gompf-Thurston symplectic form.

Let $(X,\omega)$ be a symplectic $4$--manifold and $T \subset X$ be an embedded Lagrangian torus in $X$. Take  the \emph{Lagrangian framing} $\nu T \cong T^2 \x D^2$ on the tubular neighborhood $\nu T$ of $T$, where each $T^2 \x \{ pt \}$ corresponds to a Lagrangian submanifold of $X$.  Given a curve $\gamma \subset T$ let $S^1_\gamma$ be a curve on $\partial \nu T \cong T^3$ parallel to $\gamma$ under the chosen framing. Let $\mu_T$ denote the meridian of $T$ in $\partial (\nu T)$. With these in mind, by a $(T, \gamma, p/q)$ surgery, we mean $X':= (X \setminus \nu T) \cup_\varphi (T^2 \x D^2)$, where the gluing map $\varphi\colon  \partial(T^2 \x  D^2) \to \partial (\nu T)$ satisfies $\varphi_*([\partial D^2])= p [\mu_T]+q[S^1_\gamma]$ in $H_1(X\setminus \nu T)$.  This is called a \emph{Luttinger surgery} when $p/q=1/m$, in which case the symplectic form on $X \setminus \nu T$ extends to a symplectic form $\omega'$ on $X'$ \cite{ADK}. 

Surgery along tori can change the fundamental group dramatically, which is well-exploited in the literature for symplectic $4$--manifolds containing surface products as codimension zero symplectic submanifolds. We note the following very useful technical lemma:

\begin{proposition}\label{prop:Luttinger}
Consider $X_0= \Sigma_g^1 \times \Sigma_1^1$, for $g \geq 2$,  with the product symplectic structure $\omega_0$. There are links of embedded Lagrangian tori $\mathcal{L}_i$ in $(X_0, \omega_0)$ and Luttinger surgeries along them, which result in symplectic $4$--manifolds $(X_i, \omega_i)$, for $i=1,2$, respectively, with the following property: There is a pair of curves $\{a, b\}$ on $\Sigma:=\Sigma_g^1 \times \{ pt \}$ represented by some $\{\alpha, \beta\}$ in $\pi_1(X_i)$, such that the quotient of $\pi_1(X_i)$ by the normal closure of $\beta$ is a cyclic group generated by the image of $\alpha$ under the quotient map,  where
\begin{enumerate}
[\rm{(}a\rm{)}]
\item for $i=1$, $a$ and $b$ intersect transversally once in $\Sigma$,
\item for $i=2$, $a$ and $b$ are disjoint and together do not separate $\Sigma$.
\end{enumerate}
In particular, $\pi_1(X_i)$ is the normal closure of $\{\alpha, \beta\}$.
\end{proposition}

\begin{proof}
We set our notation for the Luttinger surgeries and carry out the relevant fundamental group calculation, following the conventions in \cite{FPS,AkhmedovBaykurPark, BaldridgeKirk}. 

Consider the surface $\Sigma_g$ with the standard cell decomposition prescribed by a regular $4g$--gon with vertex $p$ and the edges labeled in the order $\Pi_{i=1}^g a_i b_i a^{-1}_i b^{-1}_i$ as we traverse the boundary. Similarly, take the standard cell decomposition of $\Sigma_1=T^2$ given by a square with vertex $q$ and  edges labeled as $x y x^{-1} y^{-1}$. Remove a small disk from the center of both to get $\Sigma_g^1$ and $\Sigma_1^1$. So $\{a_i, b_i\}_{i=1}^g$  and $\{x, y\}$ generate $\pi_1(\Sigma_g^1)$  and $\pi_1(\Sigma_1^1)$ at the base points $p$ and $q$, respectively. We calculate the fundamental group of the surgered $X_0$ at the base point $p \times q$.

For any curve $c \in \{a_i, b_i, x, y\}$, let $c', c''$ denote the two parallel copies of the curve $c$ on the same surface, as in \cite{FPS}[Figure~2]. Abusing the notation a tad, we denote any curve of the form $\gamma \x q$ or $p \x \gamma$ in $X_0=\Sigma_g^1 \x \Sigma_1^1$ also by $\gamma$. As mentioned earlier, we encode the (Luttinger) surgery information by the triple $(T, \gamma, k)$. 

We have two different configurations to consider:

\noindent{\underline{\textit{Configuration~1}}:} Perform the following Luttinger surgeries:
\[ 
(a_1' \x x', x', 1), \,
(a''_1 \x y', y', 1),  \,
(a_i' \x x', a_i', 1),  \,
(b_i' \x x'', b_i', 1)  \text{ for } i=2, \ldots, g \, .\]
Note that we have the following collection of mutually disjoint tori, where each torus intersects the link $\mathcal{L}_1$ of the above Lagrangian tori in $(X_0, \omega_0)$ transversally at one point along a distinct component:
\[ 
(b_1' \x y''), \,
(b''_1 \x x''),  \,
(b_i'' \x y'),  \,
(a_i'' \x y'')  \text{ for } i=2, \ldots, g \, .\]
Invoking the work of Baldridge and Kirk in \cite{BaldridgeKirk}, we deduce that $\pi_1(X_1)$ has a presentation with generators $\{a_1, b_1, \ldots, a_g, b_g, x, y\}$, among which the following relations hold: 
\[
 \mu \,  x = \mu' \, y  = \mu_i \, a_i = \mu'_i \,  b_i =1, \text{ for } i=2, \ldots, g,
\]
where $\mu, \mu', \mu_i, \mu'_i$ are the meridians of the surgered Lagrangian tori  given by conjugates of commutators of the pairs $\{b_1, y\}$, $\{b_1, x\}$, $\{b_i, y\}$, $\{a_i, y\}$, respectively.

When we set $b_1=1$, the first two commutators expressing $\mu$ and $\mu'$ become trivial, so the relators $\mu \, x =1$ and $\mu' \, y =1$ imply that $x=y=1$. Since the others are commutators of $y$, it follows that they are also trivial. From the remaining relators, we then deduce that $a_i=b_i=1$ for all $i=2, \ldots, g$. Hence, all the generators but the image of $a_1$ become trivial in 
$\pi_1(X_1) / N (b_1)$.

Taking the curves $a$ and $b$ in the isotopy classes of the geometrically dual pair $a_1$ and $b_1$ on $\Sigma:=\Sigma_g^1 \x q$, we complete the proof of the first case.

\smallskip
\noindent{\underline{\textit{Configuration~2}}:} This time, we perform the following Luttinger surgeries:
\[ 
(b_2'' \x x', x', 1), \,
(a'_1 \x y', y', 1),  \,
(a_i' \x y', a_i', 1),  \,
(b_j' \x y'', b_j', 1)  \text{ for }  i=3, \ldots, g \,\text{ and }j=1, \ldots, g .\]
We have the following collection of (mutually disjoint) dual tori in the complement of the link $\mathcal{L}_2$ of the above Lagrangian tori:
\[ 
(a_2'' \x y'), \,
(b''_1 \x x'),  \,
(b_i'' \x x'),  \,
(a_j'' \x x'')  \text{ for }  i=3, \ldots, g \,\text{ and }j=1, \ldots, g .\]
By \cite{BaldridgeKirk} again, we deduce that $\pi_1(X_2)$ has a presentation with generators $\{a_1, b_1, \ldots, a_g, b_g, x, y\}$, among which the following relations hold: 
\[
 \mu \,  x = \mu' \, y  = \mu_i \, a_i = \mu'_j \,  b_j =1, \text{ for } i=3, \ldots, g \text{ and }  j=1, \ldots, g \, ,
\]
where $\mu, \mu', \mu_i, \mu'_j$ are the meridians of the surgered Lagrangian tori  given by conjugates of commutators of the pairs $\{a_2, y\}$, $\{b_1, x\}$, $\{b_i, x\}$, $\{a_j, x\}$, respectively.

When we set $a_2=1$, the first commutator expressing $\mu$ becomes trivial, so the relator $\mu \, x =1$ implies that $x=1$. Since all the others are commutators of $x$, they are trivial as well. By the remaining relations above, we deduce that $y=a_i=b_j=1$ for all $i=3, \ldots, g$ and $j=1, \ldots g$. Hence, all the generators but the image of $a_1$ become trivial in 
$\pi_1(X_2) / N (a_2)$. 

This time we take the curves $a$ and $b$ in the isotopy classes of $a_1$ and $a_2$ on $\Sigma$, respectively.\end{proof}

\begin{remark} \label{rk:Luttinger}
Following \cite{BaldridgeKirk, FPS}, we could spell out the exact commutator expressions for the meridians in the above proof, but it should be evident at this point that for our calculations, there is no need for this. The surgeries in the proofs of Proposition~\ref{prop:Luttinger}(a) and (b) are generalizations of the surgeries employed in the proofs of \cite[Theorem~3]{BaykurKorkmaz} and ~\cite[Theorem~D]{BaykurHamada}, respectively; it may be helpful to pin-point to the interested reader how Luttinger surgeries in product surfaces boil down the business of killing the fundamental group via torus surgeries to killing just two generators.

Perhaps a more subtle point here is that one can see,  for instance, by examining the handle diagrams for $(X_0, \mathcal{L}_i)$, once the relators we listed suffice to show that we get an abelian group, then they generate all the remaining relators among the given generators $\{ a_1, b_1, \ldots, a_g, b_g, x, y\}$ in $\pi_1(X_i)$. This observation shows that $\pi_1(X_i)/ N(b)$ is infinite cyclic, which one can utilize effectively in constructions of new exotic $4$--manifolds with cyclic fundamental groups. 
\end{remark}


\smallskip
\subsection{Spin structures on Lefschetz fibrations over higher genera surfaces} \

An oriented $n  \geq 1$ dimensional manifold $M$ admits a spin structure if and only if the bundle $TM \oplus \epsilon^k$, where $\epsilon$ is the trivialized real line bundle, admits a trivialization over the $1$--skeleton of $M$ that extends over its $2$--skeleton.   Equivalently, when the second Stiefel-Whitney class $\omega_2(M)=0$. 
A \emph{spin structure} on $M$ is a homotopy class of such trivializations, and whenever $\omega_2(M)=0$, the set of spin structures on $M$, denoted by $\Spin(M)$, is parameterized by $H^1(M; \Z_2)$.

One can naturally assign to each $s \in \Spin(\Sigma_g)$ a \emph{quadratic form} $q \colon H_1(\Sigma_g;\Z_2) \to \Z_2$ with respect to the intersection pairing, i.e., $q(a+b)=q(a)+q(b)+a \cdot b$ for all $a, b \in H_1(\Sigma_g;\Z_2)$. This assignment gives a one-to-one correspondence between $\Spin(\Sigma_g)$ and the set of such quadratic forms. 
Given a spin structure $s \in \Spin(\Sigma_g)$, we define the \textit{spin mapping class group} $\M(\Sigma_g,s)$ to be the subgroup of $\M(\Sigma_g)$ that consists  of the stabilizers of $s$, or equivalently, of the corresponding quadratic form $q$,  in $\M(\Sigma_g)$. An elementary but useful fact to bear in mind is that a non-separating Dehn twist $t_c$ is in $\M(\Sigma_g,s)$ if and only if $q(c)=1$.

Next, we provide a sufficient condition for the total space of a Lefschetz fibration over $\Sigma_h$ to admit a spin structure, in terms of its monodromy factorization being contained in a spin mapping class group. When $h=0$, the statement below follows from Stipsicz's main theorem in \cite{StipsiczSpin} (see \cite{BaykurHayanoMonden}), and our proof is a generalization of his arguments.

\begin{proposition} \label{prop:spin}
Let $(X,f)$ be a genus $g \geq 2$ Lefschetz fibration with  a monodromy factorization 
\[ t_{c_1} \cdots t_{c_\ell} \, [A_1, B_1] \cdots [A_h, B_h] =1 \  \text{ in } \M(\Sigma_g)  \]
with each $c_i$ non-separating,  and a section of even self-intersection. If there exists a spin structure $s \in \rm{Spin}(\Sigma_g)$ such that $t_{c_i}, A_j, B_j \in \M(\Sigma_g, s)$ for all $i=1, \ldots, \ell$ and $j=1, \ldots, h$, then $X$ admits a spin structure.
\end{proposition}

\begin{proof}
Let $(X,f)$ be a Lefschetz fibration with a monodromy factorization as above, where all $t_{c_i}, A_j, B_j \in \M(\Sigma_g, s)$. We can assume that $\ell+h>0$; otherwise, the total space of a trivial $\Sigma_g$--bundle  over $S^2$ is certainly spin. 
Let $q$ be the quadratic form associated to $s$.

Let $\rho$ be the $\frac{2\pi}{\ell+4h}$--radian counter-clockwise rotation of $S^1$ and $\rm{conj} \colon S^1 \to S^1$ be the complex conjugation. Take a small interval $I$ of radians in $(0, \frac{2 \pi}{\ell+4h}) \subset S^1 = \partial D^2$.  Consider the following disjoint collection of subintervals on $S^1$:
$I_i=\rho^i(I)$, for $i=1, \ldots, \ell$, and
$J_j= \rho^{\ell+4j-3}(I)$,  $J'_j=\rho^{\ell+4j-1}(I)$,
$K_j=\rho^{\ell+4j-2}(I)$ and
$K'_j=\rho^{\ell+4j}(I)$, for $j=1, \ldots, h$.

Let $W_2$ be a regular neighborhood of a fiber, so $(X,f)$ decomposes as a Lefschetz fibration $W_1 \rightarrow \Sigma_h^1$ and $W_2$. Here $W_1$ can be obtained from the trivial fibration $W_0:=D^2 \x \Sigma_g \rightarrow D^2$ by attaching $\ell$ many (Lefschetz) $2$--handles and $2h$ copies of $D^2 \x \Sigma_g$ to $\partial W_0$ as follows: Each Lefschetz $2$--handle is attached along $\{ pt\} \x c_i \subset I_i \x \Sigma_g$ with framing one less than the fiber framing. Whereas the copies of $D^2 \x \Sigma_g$ are attached via the map  
\[
\rm{id}_{\, J_j \x \Sigma_g} \sqcup \ \rm{id}_{\, J'_{j}}\x \phi_j \colon  (J_j  \x \Sigma_g) \sqcup (J'_{j} \x \Sigma_g) \rightarrow (J_j  \x \Sigma_g) \sqcup (J'_{j} \x \Sigma_g)
\] 
or the map 
\[
\rm{id}_{\, K_j \x \Sigma_g} \sqcup \ \rm{id}_{\, K'_{j}}\x \psi_j \colon  (K_j  \x \Sigma_g) \sqcup (K'_{j} \x \Sigma_g) \rightarrow (K_j  \x \Sigma_g) \sqcup (K'_{j} \x \Sigma_g)
\]
after pre-composing with $r \times \rm{id}_{\Sigma_g}$, for some reflection $r$
of $S^1$ switching the pair of relevant intervals $\{J_j, J'_j\}$ or $\{K_j, K'_j\}$, where $\phi_j, \psi_j$ run over  the isotopy classes of $A_j, B_j$, for all $j= 1, \ldots, h$. \linebreak (Of course, we attach all of these in an order dictated by the monodromy factorization.) 

In the arguments below, we repeatedly fall back on the following observations: A submanifold $N$ of a spin manifold $M$ inherits a  spin structure when it is of codimension zero, as well as when it has trivialized normal bundle, which is always the case for the boundary or a fiber. Given any two spin manifolds $(M_i, \mathfrak{s}_i)$ with codimension zero submanifolds $N_i \subset \partial M_i$, the manifold with boundary $M_1 \cup_\varphi M_2$ given by an orientation-reversing diffeomorphism $\varphi\colon N_1 \to N_2$ admits a natural spin structure restricting to $\mathfrak{s}_i$ on each $M_i$, provided $\varphi$ is a spin map.

Now, let $\mathfrak{s}_0 \in \Spin(D^2 \x \Sigma_g)$ be the product of the unique spin structure on $D^2$ and $s$ on $\Sigma_g$. Since $q(c_i)=1$ for each $i$, as noted in \cite{StipsiczSpin}, we can extend $\mathfrak{s}_0$ on $W_0$ over any $2$--handle (with its unique spin structure) we attach along $c_i$ with one off than the fiber framing. Similarly,  since we equip each copy of $D^2 \x \Sigma_g$ we attach to $W_0$ with the same spin structure $\mathfrak{s}_0$,  we get spin attaching maps  $\rm{id}_{\, J_j \x \Sigma_g} \sqcup \ \rm{id}_{\, J'_{j}}\x \phi_j$ and $\rm{id}_{\, K_j \x \Sigma_g} \sqcup \ \rm{id}_{\, K'_{j}}\x \psi_j$. It follows that we get a spin structure $\mathfrak{s}_1$ on the union of all, which makes up $W_1$. Note that $\mathfrak{s}_1$ restricts to the fibers of the Lefschetz fibration $W_1 \to \Sigma_h^1$ as $s$, that is,  we have a spin fibration. When does it extend to one over the closed surface $\Sigma_h$?

We use Stipsicz's trick: Let $F$ denote the regular fiber of $(X,f)$ and $S$ be a section with even self-intersection. Let $N$ be a fibered neighborhood of the two, a plumbing of $D^2$--bundles over $F$ and $S$ with even Euler numbers, so $\omega_2(N)=0$. Moreover, because the intersection form on $N$ is unimodular, from the long exact sequence of the pair, one sees that $\iota^* \colon H^1(N; \Z_2) \to H^1(\partial N; \Z_2)$ is surjective, where $\iota \colon \partial N \hookrightarrow N$
is the inclusion map. Thus, when $\omega_2(N)=0$, which is the case here, any spin structure on $\partial N$ extends to one on $N$. In particular, the restriction of $\mathfrak{s}_1$ on $X \setminus N \subset W_1$ induces a spin structure on $\partial N$ that we can extend to all of $X$.
\end{proof}

\smallskip
\section{Signature zero Lefschetz fibrations over the $2$--torus}  \label{Sec:fibrations}

In this section, we are going to build signature zero symplectic Lefschetz fibrations over the $2$--torus with a few nodes,  first with non-spin, and then with spin monodromies.   The total spaces of these symplectic fibrations are the ``kernels'' of the symplectic models for the exotic signature zero $4$--manifolds we will produce in the next section.

\subsection{Non-spin families.} \

We first prove the following theorem, which constitutes (the non-spin) half of Theorem~B:

\begin{theorem} \label{thm:nonspinLFs}
For each $g \geq 4$,  there are infinitely many genus--$g$ Lefschetz fibrations over the $2$--torus with only four nodes and non-spin monodromy, whose total spaces are signature zero, non-spin, minimal symplectic \mbox{$4$--manifolds,} which are  not homotopy equivalent to each other, or to any  complex surface. 
\end{theorem}

All the fibrations in the theorem will be presented with explicit positive factorizations.

\begin{proof}
\,Here is an outline of our proof: We will first build a genus--$4$ Lefschetz fibration over the \mbox{$2$--torus} with only four nodes,  non-spin monodromy,  signature zero and  a self-intersection zero section.  To do so, we will present a positive factorization for a genus--$2$ pencil,  and then embed it into a genus--$4$ surface in a  way that will make it possible to absorb all the negative Dehn twists (which come from the boundary multi-twist in the pencil monodromy) into a single commutator in $\M(\Sigma_4^1)$.   After that,  we will vary the same construction to get an infinite family of genus--$4$ fibrations, with distinct first homology groups, and thus, with homotopy inequivalent total spaces.  Finally,  by embedding the corresponding positive factorizations into $\M(\Sigma_g)$,  we will get the promised examples for all genus $g \geq 4$.

\smallskip
\noindent
\underline{A genus--$2$ relation}:
We have the following positive factorization for a genus--$2$ pencil with four base points on the elliptic ruled surface $T^2 \twprod S^2$, which was obtained by the second author in~\cite{Hamada}:
\begin{align} \label{eq:genus2MatsumotoRelationIIB}
	\DT{B_{0,1}}\DT{B_{1,1}}\DT{B_{2,1}}\DT{C_1}
	\DT{B_{0,2}}\DT{B_{1,2}}\DT{B_{2,2}}\DT{C_2} = 
	\DT{\delta_1}\DT{\delta_2}\DT{\delta_3}\DT{\delta_4},
\end{align}
where the Dehn twist curves are as shown in Figure~\ref{fig:Genus2Matsumoto_LiftIIB}.  It is a lift of Matsumoto's genus--$2$ relation~\cite{Matsumoto}. The construction of this relation via the breeding technique is given in  our Appendix~A.

\begin{figure}[htbp]
	\centering
	\includegraphics[height=120pt]{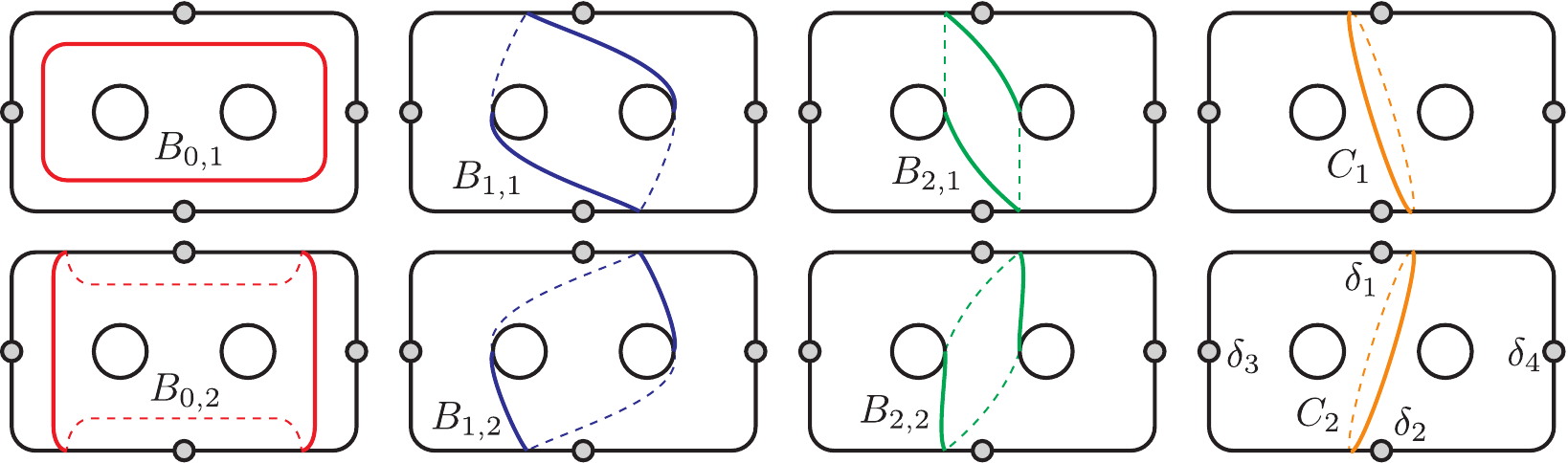}
\caption{The Dehn twist curves on $\Sigma_2^4$ for the lift of Matsumoto's genus--$2$ relation.} 
	\label{fig:Genus2Matsumoto_LiftIIB}
\end{figure}

In preparation for the next step, we modify the relation~\eqref{eq:genus2MatsumotoRelationIIB} as follows:
\begin{align*} 
	\DT{\delta_1}\DT{\delta_2}\DT{\delta_3}\DT{\delta_4} 
	&=
	\DT{B_{0,1}}\underline{\DT{B_{1,1}}\DT{B_{2,1}}\DT{C_1}
	\DT{B_{0,2}}}\DT{B_{1,2}}\DT{B_{2,2}}\DT{C_2} \\
	&=
	\underline{\DT{B_{0,1}}\DT{B_{2,1}}\DT{C_1}}
	\DT{B_{0,2}} \DT{b_1} \DT{B_{1,2}}\DT{B_{2,2}}\DT{C_2} \\
	&=
	\DT{u_2}\DT{v_2}\DT{B_{0,1}}\DT{B_{0,2}} \DT{b_1} \DT{B_{1,2}}\DT{B_{2,2}}\DT{C_2} \\
	&=
	\DT{B_{0,1}}\DT{B_{0,2}} \DT{b_1} \DT{B_{1,2}}\DT{B_{2,2}}\DT{C_2}\DT{u_2}\DT{v_2} \, ,
\end{align*}
where $b_1=t_{B_{0,2}}^{-1}t_{C_1}^{-1}t_{B_{2,1}}^{-1}(B_{1,1})$, $u_2=t_{B_{0,1}}(B_{2,1})$ and $v_2=t_{B_{0,1}}(C_1)$.
\begin{figure}[htbp]
	\centering
	\includegraphics[height=60pt]{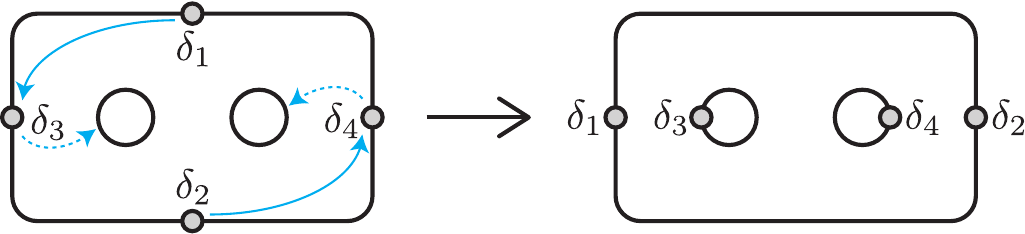}
	\caption{Dragging the boundary components of $\Sigma_2^4$.} 
	\label{fig:Genus2DraggingBoundary}
\end{figure}

By dragging the boundary components of $\Sigma_2^4$ as indicated in Figure~\ref{fig:Genus2DraggingBoundary} and relabeling the curves as $a_1:=B_{0,1}$, $a_2:=B_{0,2}$, $b_2:=B_{1,2}$, $u_1:=B_{2,2}$, and $v_1=C_2$, we have 
\begin{align} \label{eq:genus2relation}
	\DT{a_1}\DT{a_2}\DT{b_1}\DT{b_2}\DT{u_1}\DT{v_1}\DT{u_2}\DT{v_2} = 
	\DT{\delta_1}\DT{\delta_2}\DT{\delta_3}\DT{\delta_4}
\end{align}
with Dehn twist curves as shown in Figure~\ref{fig:Genus2Matsumoto_BoundingPairs}.

\begin{figure}[htbp]
	\centering
	\includegraphics[height=120pt]{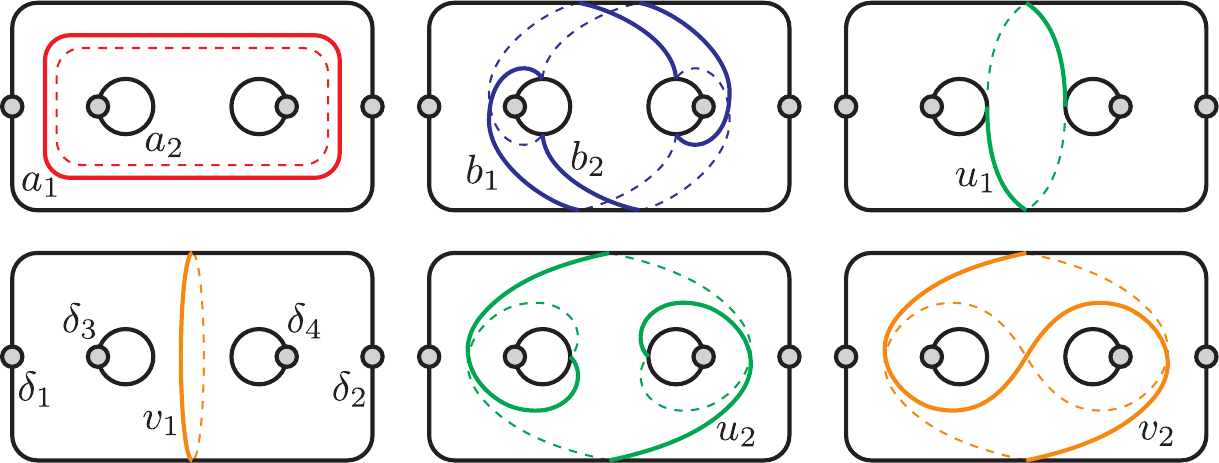}
\caption{The Dehn twist curves for the genus--$2$ pencil with two bounding pairs.} 
	\label{fig:Genus2Matsumoto_BoundingPairs}
\end{figure}
\noindent Note that  the relation~\eqref{eq:genus2relation}  contains the factors $t_{a_1}t_{a_2}$ and $t_{b_1}t_{b_2}$ which are multi-twists along the bounding  pairs $\{a_1,a_2\}$ and $\{b_1,b_2\}$.

\medskip
\noindent
\underline{A genus--$4$ Lefschetz fibration over $T^2$}:
\,We connect pairs of boundary components of the surface $\Sigma_2^4$ by adding cylinders,  while removing a disk from one of the tubes as illustrated in Figure~\ref{fig:Genus2intoGenus4Embedding}. We get a genus--$4$ surface with one boundary component.  Through such an embedding of $\Sigma_2^4$ into $\Sigma_4^1$, we can embed the relation~\eqref{eq:genus2relation} in $\M(\Sigma_2^4)$ into $\M(\Sigma_4^1)$ as
\begin{align} \label{eq:genus4relation}
	\DT{a_1}\DT{a_2}\DT{b_1}\DT{b_2}\DT{u_1}\DT{v_1}\DT{u_2}\DT{v_2}
	=
	\DT{d_1}\DT{d_2}\DT{d_3}\DT{d_4},
\end{align}
where the Dehn twist curves are given in Figure~\ref{fig:Genus4LFoverTorus}. 
Here, the images of the boundary curves $\delta_i$ are relabeled as $d_i$, whereas other curves are still denoted by the same letters.

\begin{figure}[htbp]
	\centering
	\includegraphics[height=60pt]{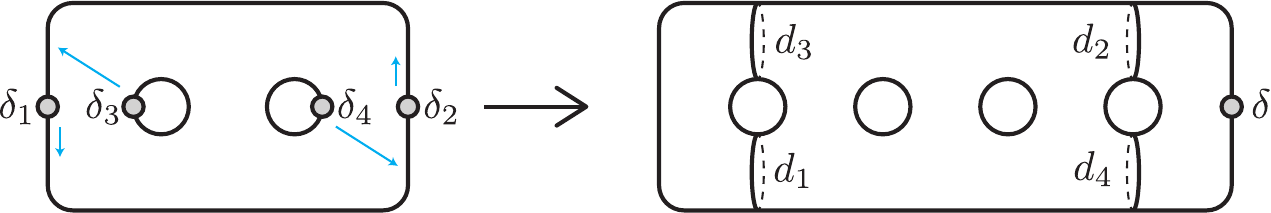}
	\caption{The embedding of $\Sigma_2^4$ into $\Sigma_4^1$; First drag boundary components of $\Sigma_2^4$ as indicated by the blue arrows. 
	Then cap the pairs $(\delta_1, \delta_3)$ and $(\delta_2,\delta_4)$ by tubes, respectively, while removing a disk from the second tube.
	}
	\label{fig:Genus2intoGenus4Embedding}
\end{figure}

\begin{figure}[htbp]
	\centering
	\includegraphics[height=200pt]{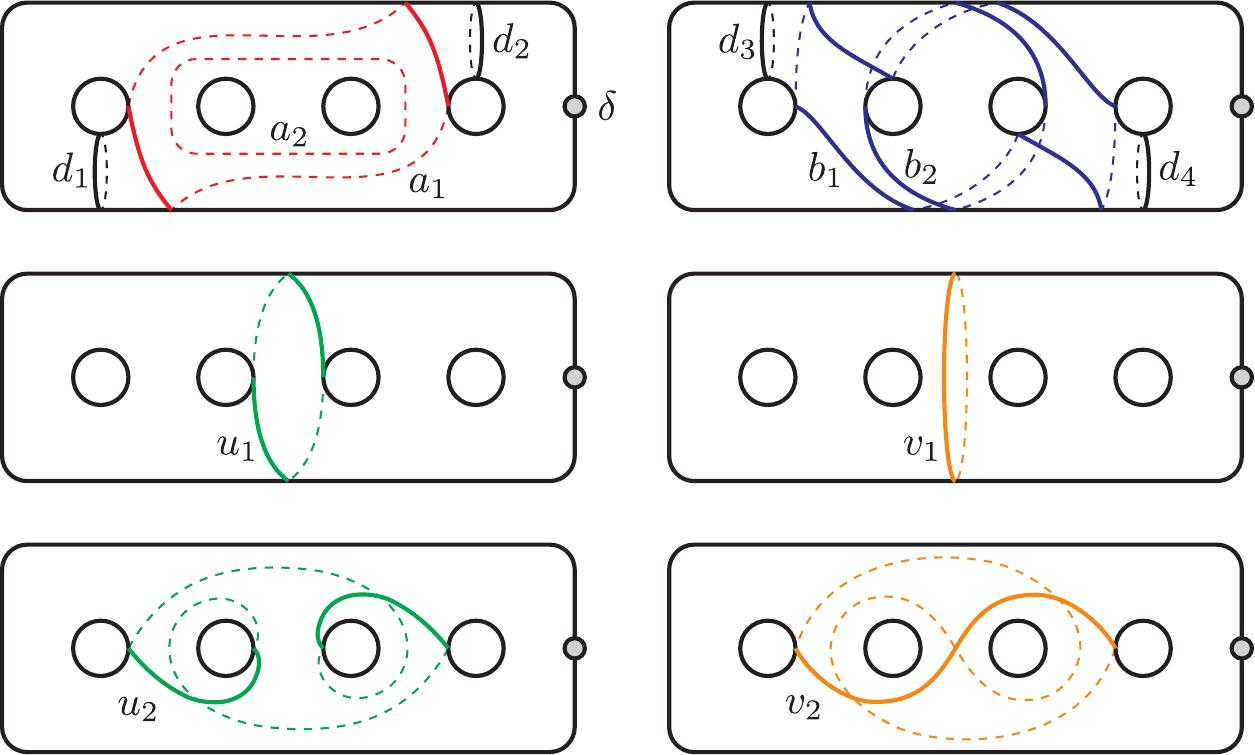}
	\caption{Embedded images of the curves for the genus--$2$ relation. The curves $u_1,v_1,u_2,v_2$ are the vanishing cycles of the Lefschetz fibration over $T^2$. }
	\label{fig:Genus4LFoverTorus}
\end{figure}

We claim that there is an orientation preserving self-diffeomorphism $\phi$ of $\Sigma_4^1$, which maps 
\[ (d_1,d_2,a_2,a_1) \mapsto (b_1,b_2,d_4,d_3). \]
This can be easily verified by looking at the complements of each quadruple of curves: Cutting $\Sigma_4^1$ along $\{d_1,d_2,a_2,a_1\}$ results in two disjoint surfaces $S_1 \cong \Sigma_0^4$ and $S_2 \cong \Sigma_1^5$. Similarly, cutting $\Sigma_4^1$ along $\{b_1,b_2,d_4,d_3\}$ yields two disjoint surfaces $S_1^\prime \cong \Sigma_0^4$ and $S_2^\prime \cong \Sigma_1^5$. All four subsurfaces have orientations induced  from $\Sigma_4^1$.
Now, there is an orientation-preserving diffeomorphism between $S_1$ and $S_1^\prime$ that exchanges the boundary components as we like. In particular, we can take it so that the ordered quadruple of the curves $(d_1,d_2,a_2,a_1)$ maps to $(b_1,b_2,d_4,d_3)$.
Likewise, we can find an orientation-preserving diffeomorphism between $S_2$ and $S_2^\prime$ that maps $(d_1,d_2,a_2,a_1)$ to $(b_1,b_2,d_4,d_3)$.
Combining the two diffeomorphisms, with smooth adjustment around $(b_1,b_2,d_4,d_3)$, gives us the desired diffeomorphism of $\Sigma_4^1$.

Let us fix such a diffeomorphism $\phi$, and denote its isotopy class in $\M(\Sigma_4^1)$ by the same letter.
We can now modify the relation~\eqref{eq:genus4relation} as follows:
\begin{align*}
	1&=
	\DT{a_1}\DT{a_2}\LDT{d_2}\LDT{d_1} \cdot \DT{b_1}\DT{b_2}\LDT{d_4}\LDT{d_3} 
	\cdot 
	\DT{u_1}\DT{v_1}\DT{u_2}\DT{v_2} \\
	&=
	\DT{a_1}\DT{a_2}\LDT{d_2}\LDT{d_1} \cdot \DT{\phi(d_1)}\DT{\phi(d_2)}
	\LDT{\phi{(a_2)}}\LDT{\phi{(a_1)}} \cdot 
	\DT{u_1}\DT{v_1}\DT{u_2}\DT{v_2} \\
	&=
	\DT{a_1}\DT{a_2}\LDT{d_2}\LDT{d_1} \cdot \phi \ \DT{d_1}\DT{d_2}\LDT{a_2}
	\LDT{a_1} \phi^{-1}
	\cdot 
	\DT{u_1}\DT{v_1}\DT{u_2}\DT{v_2} \\
	&=
	[\DT{a_1}\DT{a_2}\LDT{d_2}\LDT{d_1}, \phi] 
	\cdot 
	\DT{u_1}\DT{v_1}\DT{u_2}\DT{v_2} \\
	&=
	\DT{u_1}\DT{v_1}\DT{u_2}\DT{v_2} [\DT{a_1}\DT{a_2}\LDT{d_2}\LDT{d_1}, \phi].
\end{align*}
For $\DT{\delta}$ denoting the Dehn twist along a boundary parallel curve, we have:
\begin{align} \label{eq:genus4LFoverTorusMonodromy}
	\DT{u_1}\DT{v_1}\DT{u_2}\DT{v_2} [\DT{a_1}\DT{a_2}\LDT{d_2}\LDT{d_1}, \phi] 
	 = \DT{\delta}^0 \, ,
\end{align}
which prescribes a genus--$4$ Lefschetz fibration over $T^2$, with four nodes  corresponding to the vanishing cycles $u_1,v_1,u_2,v_2$, and a section of self-intersection zero, corresponding to the trivial power of the boundary twist. 
Note that $v_1$ is a separating curve,  whereas $u_1, u_2$ and $v_2$ are not.

\medskip
\noindent
\underline{An infinite family of fibrations over $T^2$}:
We can vary our choice of the mapping class $\phi$ in the relation~\eqref{eq:genus4LFoverTorusMonodromy} in order to generate an infinite family of genus--$4$ Lefschetz fibrations over $T^2$ with the same properties.  This would also allow us to describe such $\phi$ explicitly.

For $n \in \mathbb{Z}$ we define $\phi_n \in \M(\Sigma_4^1)$ as
\begin{align*}
	\phi_n := (t_{2})^{n+1} t_{9} t_{5}^{-1}  t_{8}^{-1}  t_{7}^{-1}  t_{6}^{-1}  t_{1} t_{5} t_{4}^{-1}  t_{3}  t_{2}^{-1}  t_{1}^{-1} 
\end{align*}
where the Dehn twists $t_1, \ldots, t_9$ are along the curves labeled by $1, \ldots, 9$  in Figure~\ref{fig:Genus4LFoverTorus_Twistings}. By successively applying the Dehn twists in the definition of $\phi_n$, we see that $\phi_n$ maps $(d_1,d_2,a_2,a_1)$ to $(b_1,b_2,d_4,d_3)$. Note that the last factor $t_2^{n+1}$ is supported away from $ \{b_1,b_2,d_4,d_3\}$.

\begin{figure}[htbp]
	\centering
	\includegraphics[height=55pt]{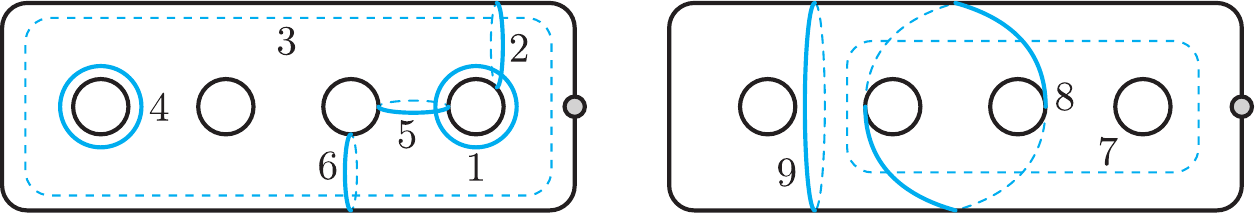}
	\caption{Dehn twist curves for defining $\phi_n$.}
	\label{fig:Genus4LFoverTorus_Twistings}
\end{figure}

So, in the relation~\eqref{eq:genus4LFoverTorusMonodromy} we can  replace $\phi$ by $\phi_n$,  and get a positive factorization
\begin{align} \label{eq:genus4LFoverTorusMonodromyInfinite}
	\DT{u_1}\DT{v_1}\DT{u_2}\DT{v_2} [\DT{a_1}\DT{a_2}\LDT{d_2}\LDT{d_1}, \phi_n ] 
	= \DT{\delta}^0 \, 
\end{align} 
in $\M(\Sigma_4^1)$,  which prescribes a genus--$4$ Lefschetz fibration $f_n \colon X_n \to T^2$,  with a section $s_n \subset X_n$ of self-intersection zero,  for each $n \in \Z$.

Moreover,  using any embedding  $\Sigma_4^1 \hookrightarrow \Sigma_g^1$,  we can view the positive factorizations
~\eqref{eq:genus4LFoverTorusMonodromyInfinite}  in $\M(\Sigma_g^1)$ instead,  for each $g \geq 4$.  Each one of these factorizations in $\M(\Sigma_g^1)$ now prescribes a genus--$g$ Lefschetz fibration \mbox{$f_{g,n} \colon X_{g,n} \to T^2$} with a section $s_{g,n} \subset X_{g,n}$ of self-intersection zero.  In this notation  $(X_n, f_n)=(X_{4,n}, f_{4,n})$.

For each fixed $g \geq 4$,  the total spaces in the family $\{ X_{g,n} \, | \, n \in \N \}$ will be pairwise homotopy inequivalent.  This  follows from them having distinct first homology groups, as we will show next.

\medskip
\noindent
\underline{First homology calculation}: Let us first compute  the first integral homology group of the total space $X_n$ of the genus--$4$ fibration $f_n$,  which we will do so using its monodromy factorization~\eqref{eq:genus4LFoverTorusMonodromyInfinite}.

\begin{figure}[htbp]
	\centering
	\includegraphics[height=65pt]{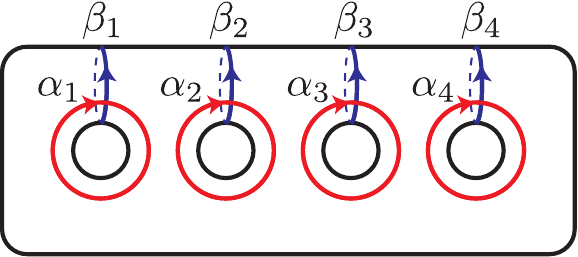}
	\caption{The generators for $H_1(\Sigma_4)$.}
	\label{fig:Genus4H1generators}
\end{figure}

Fix a geometric basis $\mathcal{B}:=\{ \alpha_1, \beta_1, \cdots, \alpha_4, \beta_4 \}$ for $H_1(\Sigma_4)$, where the oriented curves are as shown  in Figure~\ref{fig:Genus4H1generators}.
Then,  $H_1(X_n)$ is isomorphic to the quotient of $H_1(T^2) \oplus H_1(\Sigma_4)$ modulo the following types of relations induced by the monodromy of the fibration:
\begin{enumerate}[(i)]
	\item
$\eta=0$ for each  representative $\eta$ of the Lefschetz vanishing cycles $u_1, v_1, u_2$ and $v_2$ (taken with auxiliary orientations) expressed in the basis $\mathcal{B}$,  and
	\item
	$\DT{a_1}\DT{a_2}\LDT{d_2}\LDT{d_1}(\gamma) = \gamma$ and $\phi_n(\gamma) =\gamma$  for each $\gamma \in \mathcal{B}$.
\end{enumerate}

The latter can be calculated using the Picard-Lefschetz formula.  The following hold in $H_1(\Sigma_4)$:
\begin{align*}
	&u_1 = \beta_2 + \beta_3; &
	&v_1 = 0; \\
	&u_2 = -2\alpha_2 -2\alpha_3 +\beta_1 -\beta_2 -\beta_3 + \beta_4; &
	&v_2 = \beta_1 -\beta_4; \\  	
	&\DT{a_1}\DT{a_2}\LDT{d_2}\LDT{d_1}(\alpha_1) = \alpha_1 + \alpha_2 + \alpha_3 -\beta_4; & 
	&\DT{a_1}\DT{a_2}\LDT{d_2}\LDT{d_1}(\beta_1) = \beta_1; \\
	&\DT{a_1}\DT{a_2}\LDT{d_2}\LDT{d_1}(\alpha_2) = \alpha_2; &
	&\DT{a_1}\DT{a_2}\LDT{d_2}\LDT{d_1}(\beta_2) = 2\alpha_2 +2\alpha_3 -\beta_1 +\beta_2 -\beta_4; \\
	&\DT{a_1}\DT{a_2}\LDT{d_2}\LDT{d_1}(\alpha_3) = \alpha_3; &
	&\DT{a_1}\DT{a_2}\LDT{d_2}\LDT{d_1}(\beta_3) = 2\alpha_2 +2\alpha_3 -\beta_1 +\beta_3 -\beta_4; \\
	&\DT{a_1}\DT{a_2}\LDT{d_2}\LDT{d_1}(\alpha_4) = \alpha_2 + \alpha_3 + \alpha_4 -\beta_1; &
	&\DT{a_1}\DT{a_2}\LDT{d_2}\LDT{d_1}(\beta_4) = \beta_4; \\
	&\phi_n(\alpha_1) = \alpha_1; &
	&\phi_n(\beta_1) = -\alpha_2 -\alpha_3 +\beta_1 +\beta_2 +\beta_3 +\beta_4; \\
	&\phi_n(\alpha_2) = -\alpha_3 +\beta_2 +\beta_4; &
	&\phi_n(\beta_2) = \alpha_1 -\alpha_2 -\alpha_3 -\alpha_4 +\beta_2 +\beta_3 +(n+1)\beta_4; \\
	&\phi_n(\alpha_3) = \alpha_3 -\beta_2; &
	&\phi_n(\beta_3) = \alpha_1 -\alpha_2 -\alpha_3 -\alpha_4 +2\beta_3 +(n+1)\beta_4; \\
	&\phi_n(\alpha_4) = \alpha_1 -\alpha_2 -\alpha_3 +\beta_3 +\beta_4; &
	&\phi_n(\beta_4) = \alpha_2 +\alpha_3 -\beta_2 -\beta_3.
\end{align*}
Hence, we obtain the following nontrivial relations:
\begin{align}
	&\beta_2 + \beta_3 =0; \label{rel:genus4LF1}\\ 
	&-2\alpha_2 -2\alpha_3 +\beta_1 -\beta_2 -\beta_3 + \beta_4 =0; \label{rel:genus4LF2}\\
	&\beta_1 -\beta_4 =0; \label{rel:genus4LF3}\\
	&\alpha_1 + \alpha_2 + \alpha_3 -\beta_4 = \alpha_1; \label{rel:genus4LF4}\\
	&\alpha_2 + \alpha_3 + \alpha_4 -\beta_1 =\alpha_4; \label{rel:genus4LF5}\\ 
	&2\alpha_2 +2\alpha_3 -\beta_1 +\beta_2 -\beta_4 = \beta_2; \label{rel:genus4LF6}\\
	&2\alpha_2 +2\alpha_3 -\beta_1 +\beta_3 -\beta_4 = \beta_3; \label{rel:genus4LF7}\\
	&-\alpha_3 +\beta_2 +\beta_4 =\alpha_2; \label{rel:genus4LF8}\\
	&\alpha_3 -\beta_2 =\alpha_3; \label{rel:genus4LF9}\\
	&\alpha_1 -\alpha_2 -\alpha_3 +\beta_3 +\beta_4 =\alpha_4; \label{rel:genus4LF10}\\
	&-\alpha_2 -\alpha_3 +\beta_1 +\beta_2 +\beta_3 +\beta_4 =\beta_1; \label{rel:genus4LF11}\\
	&\alpha_1 -\alpha_2 -\alpha_3 -\alpha_4 +\beta_2 +\beta_3 +(n+1)\beta_4 =\beta_2; \label{rel:genus4LF12}\\
	&\alpha_1 -\alpha_2 -\alpha_3 -\alpha_4 +2\beta_3 +(n+1)\beta_4 =\beta_3; \label{rel:genus4LF13}\\
	&\alpha_2 +\alpha_3 -\beta_2 -\beta_3 = \beta_4. \label{rel:genus4LF14}
\end{align}
The relation~\eqref{rel:genus4LF9} gives $\beta_2=0$.
Then \eqref{rel:genus4LF1} yields $\beta_3=0$.
We have $\beta_1 =\beta_4$ from \eqref{rel:genus4LF3}, $\alpha_3 = -\alpha_2 +\beta_4$ from \eqref{rel:genus4LF4}.
By substituting $\beta_3=0$ and $\alpha_3 = -\alpha_2 +\beta_4$ in \eqref{rel:genus4LF10}, we obtain $\alpha_4=\alpha_1$.
Now we substitute $\beta_3=0$, $\alpha_3 = -\alpha_2 +\beta_4$, and $\alpha_4 = \alpha_1$ in \eqref{rel:genus4LF13}, which results in $n\beta_4=0$.
So we have the simplified relations
\begin{align*}
	\alpha_3 &= -\alpha_2 +\beta_4; \\
	\alpha_4 &= \alpha_1; \\
	\beta_1 &= \beta_4; \\
	\beta_2 &= 0; \\
	\beta_3 &= 0; \\
	n\beta_4 &=0.
\end{align*}
and in fact,  the above six relations generate all of the relations~\eqref{rel:genus4LF1}--\eqref{rel:genus4LF14}.
Therefore,
\begin{align}
	H_1(X_n) \cong H_1(T^2) \oplus \mathbb{Z}\left< \alpha_1, \alpha_2, \beta_4 \;|\; n\beta_4=0 \right> \cong  \mathbb{Z}^4 \oplus \mathbb{Z}/n\mathbb{Z}.
\end{align}

Finally, since the monodromy action of $\pi_1(T^2)$ on $H_1(\Sigma_g)$ is a trivial extension of its action on $H_1(\Sigma_4^1)$,  it follows from the  homology calculation above that 
\[ H_1(X_{g,n}) \cong \mathbb{Z}^{2g-4} \oplus \mathbb{Z}/n\mathbb{Z}.  \] 

\medskip
\noindent
\underline{Algebraic topology of $X_{g,n}$}: It is fairly easy to calculate several topological invariants of $X_{g,n}$ using the explicit positive factorization for $f_{g,n}$.  The Euler characteristic of $X_{g,n}$ is  given by 
\[\eu(X_{g,n})=\eu(\Sigma_g) \, \eu(T^2) + \ell= 0+ 4 = 4, \]
where $\ell$ is the number of nodes,  or equivalently, the number of Dehn  twists in the factorization.  

By the work of Endo and Nagami \cite{EndoNagami},  we can calculate the signature of $X_{g,n}$ by an algebraic count of the relations we have employed to derive the final positive factorization for the fibration $f_{g,n}$ in  $\M(\Sigma_g)$.  Embeddings of relations,  conjugations,  Hurwitz moves,  or re-expressing a subword as a commutator do not change the signature.  It follows that the signature of the positive factorization for $f_{g,n}$ is equal to the signature of the genus--$2$ relation~\eqref{eq:genus2MatsumotoRelationIIB} we started with.  Since the latter corresponds to a pencil on $T^2 \twprod S^2$,  which has  signature zero,  we conclude that $\sigma(X_{g,n})=0$. 

The remaining homology groups of the closed,  connected,  oriented $4$--manifold $X_{g,n}$,  as well as $b_2^+$ and $b_2^-$ are then determined by its Euler characteristic,  signature and first homology,  via Poincar\'{e} duality and the universal coefficient theorem.  In particular,  we get $H_2(X_{g,n})= \Z^{4g-6} \oplus \,\Z / n \Z$ and $b_2^{\pm}(X_{g,n})= 2g-3$ for $n \neq 0$,  whereas $H_2(X_{g,0})= \Z^{4g-8}$ and $b_2^{\pm}(X_{g,0})=2g-4$.  

Lastly,  since the monodromy factorization of $(X_{g,n}, f_{g,n})$ contains a Dehn twist along the separating curve $v_1$ (which remains separating after the embedding $\Sigma_4^1 \hookrightarrow \Sigma_g^1$),  the  components of the corresponding reducible fiber are surfaces of self-intersection $-1$.  Therefore,  $X_{g,n}$ has an odd intersection form and does not admit any spin structure. 

\medskip
\noindent
\underline{Differential topology of $X_{g,n}$}: We can equip  $(X_{g,n}, f_{g,n})$ with a Gompf--Thurston symplectic form,  with respect to which,   the fibers and the section $s_{g,n}$ are symplectic.  Since any relatively minimal Lefschetz fibration over a positive genus surface is minimal \cite{StipsiczMinimal}, each $X_{g,n}$ is a minimal symplectic $4$--manifold.  We have 
\[c_1^2(X_{g,n})= 2\,\eu(X_{g,n})+3\,  \sigma(X_{g,n})= 8 \ \text{ and } \  \chi_h(X_{g,n})=\frac{1}{4}(\eu(X_{g,n})+ \sigma(X_{g,n}))=1 \]
 for any $g \geq 4$ and  $n \in \Z$.  Since $X_{g,n}$ is minimal and  $c_1^2(X_{g,n})>0$,  the symplectic Kodaira dimension $\kappa(X_{g,n})=2$,  i.e.  $X_{g,n}$ is a minimal symplectic $4$--manifold of general type.  

For each fixed $g \geq 4$,  the total spaces in the family $\{ X_{g,n} \, | \, n \in \N \}$ are pairwise homotopy inequivalent,  because they have distinct first homology groups.  
Since there are only finitely many deformation  classes of compact complex surfaces of general type with the same $c_1^2$ and $\chi_h$  \cite{Gieseker}, we conclude that for all but finitely many $n \in \N$,   $X_{g,n}$ does not have the homotopy type of a complex surface.  Passing to an infinite countable subset of $\N$ if necessary,  and  relabeling the index set as $\N$ again,  we can then get an infinite family of minimal symplectic $4$--manifolds $\{X_{g,n} \,  | \, n \in \N \}$, where none of which are homotopy equivalent to any compact complex surface. 
\end{proof}

\smallskip
\subsection{Spin families} \

The next theorem constitutes the second (spin) half of Theorem ~B.

\begin{theorem} \label{thm:spinLFs}
For each $g \geq 5$,  there are infinitely many genus--$g$ Lefschetz fibrations over the $2$--torus with only four nodes and spin monodromy, whose total spaces are signature zero, spin, symplectic \mbox{$4$--manifolds,} which are  not homotopy equivalent to each other, or to any compact complex surface.
\end{theorem}

As before,  all fibrations  in the theorem will be presented with explicit positive factorizations.  We will obtain our examples by refining a construction given by the second author in ~\cite{HamadaMinimal} to attain  \emph{spin} monodromies. 

\begin{proof}
\,We will first build a genus--$5$ Lefschetz fibration over the \mbox{$2$--torus} with only four nodes,  spin monodromy,  signature zero and  a self-intersection zero section, following \cite{HamadaMinimal}.  For this construction, we will take a positive factorization for a genus--$1$ pencil on the spin manifold $S^2 \x S^2$ with eight base points, and embed it into a genus--$5$ surface so that we can express all the negative Dehn twists (coming from the boundary multi-twist) and eight of the positive Dehn twists as a single commutator in $\M(\Sigma_5^1)$.  We will make all these choices carefully so that the resulting monodromy is spin.  Afterward, we will vary the same construction to get an infinite family of genus--$5$ spin fibrations, with homotopy inequivalent total spaces that can be distinguished by their first homology groups, and then get the higher genera examples by embedding the corresponding positive factorizations into $\M(\Sigma_g)$,  for each $g \geq 5$.  

\smallskip
\noindent
\underline{An eight-holed torus relation}:
We begin with the following version of an \textit{eight-holed torus relation} in \cite{HamadaMinimal}, where the Dehn twist curves are as depicted in Figure~\ref{fig:8HoledTorus}:
\begin{align} \label{eq:8holed1}
	\DT{c_1}\DT{c_2}\DT{c_3}\DT{c_4}\DT{\sigma_1}\DT{\sigma_2}\DT{\sigma_3}
	\DT{\sigma_4}\DT{d^{\prime\prime\prime}}\DT{d^{\prime\prime}}\DT{d^{\prime}}
	\DT{d}
	=
	\DT{\delta_1}\DT{\delta_1^\prime}\DT{\delta_2}\DT{\delta_2^\prime}\DT{\delta_3}\DT{\delta_3^\prime}
	\DT{\delta_4}\DT{\delta_4^\prime},
\end{align}
where $d^\prime = t_{e_1}^{-1}t_{e_3}^{-1}(d)$, $d^{\prime\prime} =  t_{e_2}^{-1}t_{e_4}^{-1} t_{e_1}^{-1}t_{e_3}^{-1}(d)$, and $d^{\prime\prime\prime} =  t_{e_2}^{-1}t_{e_4}^{-1} t_{e_1}^{-2}t_{e_3}^{-2}(d)$.
To make the curves less tangled we globally conjugate the positive factorization~\eqref{eq:8holed1} by $t_{e_1}t_{e_3}$ and get
\begin{align} \label{eq:8holed2}
	\DT{c_1}\DT{c_2}\DT{c_3}\DT{c_4}\DT{\sigma_1}\DT{\sigma_2}\DT{\sigma_3}
	\DT{\sigma_4}
	\DT{d_3}\DT{d_2}\DT{d}\DT{d_1}
	=
	\DT{\delta_1}\DT{\delta_1^\prime}\DT{\delta_2}\DT{\delta_2^\prime}\DT{\delta_3}\DT{\delta_3^\prime}
	\DT{\delta_4}\DT{\delta_4^\prime},
\end{align}
where $d_1=t_{e_1}t_{e_3}(d)$, $d_2= t_{e_2}^{-1}t_{e_4}^{-1}(d)$, and $d_3 = t_{e_1}^{-1}t_{e_2}^{-1}t_{e_3}^{-1}t_{e_4}^{-1}(d)$.

This eight-holed torus relation was derived using a single $2$--chain relation and seven lantern relations, so the signature of the positive factorization is $-7 + 7\cdot1 =0$.  While not necessary for our arguments to follow, one can moreover see that this is a pencil on  $S^2 \x S^2$:  The total space of any genus--$1$ Lefschetz \emph{pencil} is birationally equivalent to $\CP$, where a simple  Euler characteristic calculation shows that the total space of this pencil is either  $\CP \# \CPb$ or $S^2 \x S^2$.  Using~\cite{BaykurHayanoMonden}[Theorem~5.1] one can then check that the total space is the spin $4$--manifold $S^2 \x S^2$.\footnote{Whereas the eight-holed torus relation of Korkmaz and Ozbagci in \cite{KorkmazOzbagci} is in fact a pencil on $\CP \# \CPb$, so the two are certainly not Hurwitz equivalent; cf.\cite{HamadaHayano}.}

\begin{figure}[htbp]
	\centering
	\includegraphics[height=120pt]{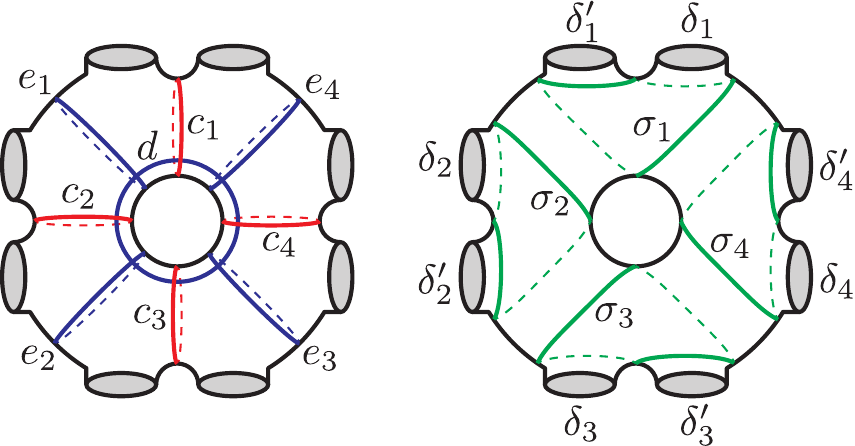}
	\caption{The Dehn twist curves for the eight-holed torus relation.} 
	\label{fig:8HoledTorus}
\end{figure}
\begin{figure}[htbp]
	\centering
	\includegraphics[height=135pt]{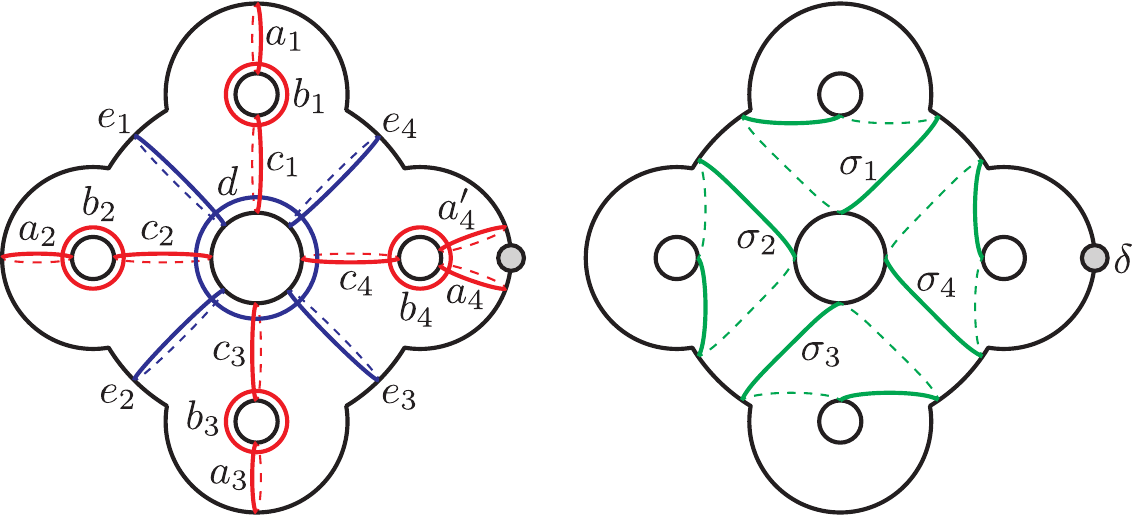}
	\caption{The embedding of the eight-holed  torus into the genus--$5$ surface (with one boundary) and the Dehn twist curves of the embedded relation.} 
	\label{fig:Genus5LFoverTorus}
\end{figure}

\smallskip
\noindent
\underline{A genus--$5$ Lefschetz fibration over $T^2$}:
We embed the eight-holed torus into $\Sigma_5^1$ by connecting each pair of boundary components $\delta_i$ and $\delta_i^\prime$ by  a cylinder for $i=1,2,3,4$ and then removing a disk from the  cylinder connecting $\delta_4$ and $\delta_4^\prime$ as illustrated in Figure~\ref{fig:Genus5LFoverTorus}.  The images of the curves of the eight-holed torus relation through this embedding can also be viewed in the same figure.  Thus, in $\M(\Sigma_5^1)$ we have
\begin{align} \label{eq:8holedtorusInGenus5}
	\DT{c_1}\DT{c_2}\DT{c_3}\DT{c_4}\DT{\sigma_1}\DT{\sigma_2}\DT{\sigma_3}
	\DT{\sigma_4}
	\DT{d_3}\DT{d_2}\DT{d}\DT{d_1}
	=
	\DT{a_1}^2\DT{a_2}^2\DT{a_3}^2\DT{a_4}\DT{a_4^\prime}.
\end{align}

We define $\tilde{\phi}_i := (t_{b_i}t_{a_i})^3(t_{c_i}t_{b_i}t_{a_i})^2$ for each 
$i=1,2,3$ and $\tilde{\phi}_4 := t_{b_4}t_{a_4^\prime}t_{b_4}t_{a_4}t_{b_4}t_{a_4^\prime}(t_{c_4}t_{b_4}t_{a_4})^2$ then $\phi := \tilde{\phi}_1\tilde{\phi}_2\tilde{\phi}_3\tilde{\phi}_4$.
It is easy to verify that $\phi$ maps $(a_i,c_i)$ to $(\sigma_i,a_i)$ for $i=1,2,3$ and $(a_4,c_4)$ to $(\sigma_4,a_4^\prime)$.
Then we modify the relation~\eqref{eq:8holedtorusInGenus5} as follows:
\begin{align*}
	1&=
	\DT{d_3}\DT{d_2} \DT{d}\DT{d_1} \cdot
	\DT{c_1}\DT{c_2}\DT{c_3}\DT{c_4} \LDT{a_4}\LDT{a_3}\LDT{a_2}\LDT{a_1}
	\cdot
	\DT{\sigma_1}\DT{\sigma_2}\DT{\sigma_3}\DT{\sigma_4} \LDT{a_4^\prime}\LDT{a_3}
	\LDT{a_2}\LDT{a_1} \\
	&= 
	\DT{d_3}\DT{d_2}\DT{d}\DT{d_1} \cdot
	\DT{c_1}\DT{c_2}\DT{c_3}\DT{c_4} \LDT{a_4}\LDT{a_3}\LDT{a_2}\LDT{a_1}
	\cdot
	\DT{\phi(a_1)}\DT{\phi(a_2)}\DT{\phi(a_3)}\DT{\phi(a_4)} \LDT{\phi(c_4)}
	\LDT{\phi(c_3)}\LDT{\phi(c_2)}\LDT{\phi(c_1)} \\
	&=
	\DT{d_3}\DT{d_2}\DT{d}\DT{d_1} \cdot
	\DT{c_1}\DT{c_2}\DT{c_3}\DT{c_4} \LDT{a_4}\LDT{a_3}\LDT{a_2}\LDT{a_1}
	\cdot
	\phi \,
	\DT{a_1}\DT{a_2}\DT{a_3}\DT{a_4} \LDT{c_4}\LDT{c_3}\LDT{c_2}\LDT{c_1}
	\phi^{-1} \\
	&=
	\DT{d_3}\DT{d_2}\DT{d}\DT{d_1} \cdot
	[\DT{c_1}\DT{c_2}\DT{c_3}\DT{c_4} \LDT{a_4}\LDT{a_3}\LDT{a_2}\LDT{a_1}, 
	\phi \,], 
\end{align*}
through commutativity relations, Hurwitz moves and re-expressing a subword as a commutator. 
Thus, we obtain the relation 
\begin{align} \label{eq:genus5LFovertorus_v1}
	\DT{d_3}\DT{d_2} \DT{d}\DT{d_1} 
	[\DT{c_1}\DT{c_2}\DT{c_3}\DT{c_4} \LDT{a_4}\LDT{a_3}\LDT{a_2}\LDT{a_1}, 
	\phi]
	 = \DT{\delta}^0
\end{align}
in $\M(\Sigma_5^1)$, which prescribes a genus--$5$ Lefschetz fibration $h \colon Y \to T^2$ with four nodes and a section of self-intersection zero.
Since we derived this positive factorization (using relations that do not affect the signature) from the eight-holed torus relation, which has signature zero,  we conclude the total space of our Lefschetz fibration has signature zero as well.

In order to have favorable vanishing cycles for the purposes of the next section, we will make one more modification to this monodromy. 
First, observe the following simple fact.
\begin{lemma} \label{lem:commutatortrick}
	Suppose that $G$ is a group and $x, A, B \in G$. Then, the relation
	$$[A,B]x=x^A [A,x^{-1}B]$$ 
	hold in $G$.
\end{lemma}
By a cyclic permutation, we rewrite the relation~\eqref{eq:genus5LFovertorus_v1} as 
\begin{align*} 
	\DT{d_2} \DT{d}\DT{d_1} 
	[\DT{c_1}\DT{c_2}\DT{c_3}\DT{c_4} \LDT{a_4}\LDT{a_3}\LDT{a_2}\LDT{a_1}, 
	\phi]
	\DT{d_3}
	= \DT{\delta}^0
\end{align*}
and apply Lemma~\ref{lem:commutatortrick} to obtain
\begin{align} \label{eq:genus5LFovertorus_spin}
	\DT{d_2}\DT{d}\DT{d_1}\DT{\tilde{d}} \;
	[\DT{c_1}\DT{c_2}\DT{c_3}\DT{c_4} \LDT{a_4}\LDT{a_3}\LDT{a_2}\LDT{a_1}, 
	\LDT{d_3}\phi]
	 = \DT{\delta}^0,
\end{align}
where 
\begin{align*}
	\tilde{d} 
	&=t_{c_1}t_{c_2}t_{c_3}t_{c_4} 
	t_{a_4}^{-1}t_{a_3}^{-1}t_{a_2}^{-1}t_{a_1}^{-1}(d_3) \\
	&=t_{c_1}t_{c_2}t_{c_3}t_{c_4} 
	t_{a_4}^{-1}t_{a_3}^{-1}t_{a_2}^{-1}t_{a_1}^{-1} 
	t_{e_1}^{-1}t_{e_2}^{-1}t_{e_3}^{-1}t_{e_4}^{-1}(d) \\
	&=t_{c_1}t_{c_2}t_{c_3}t_{c_4} 
	t_{e_1}^{-1}t_{e_2}^{-1}t_{e_3}^{-1}t_{e_4}^{-1}(d).
\end{align*}
The curve $\tilde{d}$ on the closed surface $\Sigma_5$ is depicted in Figure~\ref{fig:spinLFcurves}, together with $d$, $d_1$ and $d_2$. 
Notice that the vanishing cycles $d$ and $\tilde{d}$ are disjoint but not a bounding pair. One can see that the change from the monodromy factorization~\eqref{eq:genus5LFovertorus_v1} to~\eqref{eq:genus5LFovertorus_spin} is a variation of Hurwitz move that corresponds to a change in choices of paths in the Hurwitz system, hence this modification does not alter the Lefschetz fibration $h \colon Y \to T^2$.

\begin{figure}[htbp]
	\centering
	\includegraphics[height=100pt]{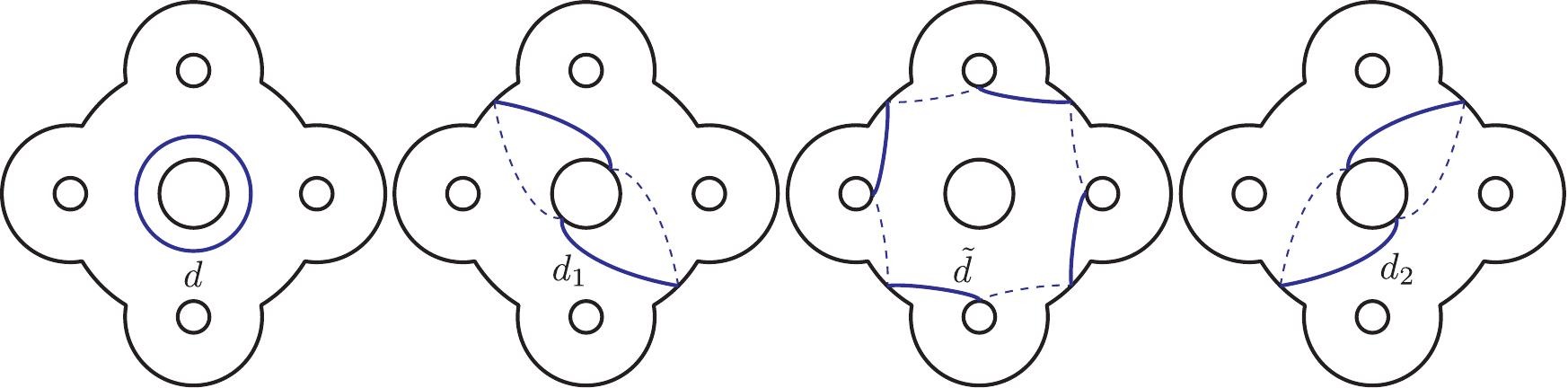} 
	\caption{The curves $d, d_1, \tilde{d}, d_2$ are the vanishing cycles of the spin Lefschetz fibration.} 
	\label{fig:spinLFcurves}
\end{figure}

\noindent
\underline{Spin condition}:
Now we find a quadratic form $q$ on $H_1(\Sigma_5;\mathbb{Z}_2)$ such that $q$ 
assigns $1$ to all the vanishing cycles, whereas each mapping class entry in  the commutator 
preserves $q$.
Consider the geometric basis $\mathcal{B}':=\{ \alpha_1,\beta_1\, \dots, \alpha_5,\beta_5 \}$ for $H_1(\Sigma_5;\mathbb{Z})$ given by the oriented curves in Figure~\ref{fig:Genus5H1basis}, which also yield a symplectic basis for  $H_1(\Sigma_5;\mathbb{Z}_2)$.

\begin{figure}[htbp]
	\centering
	\includegraphics[height=120pt]{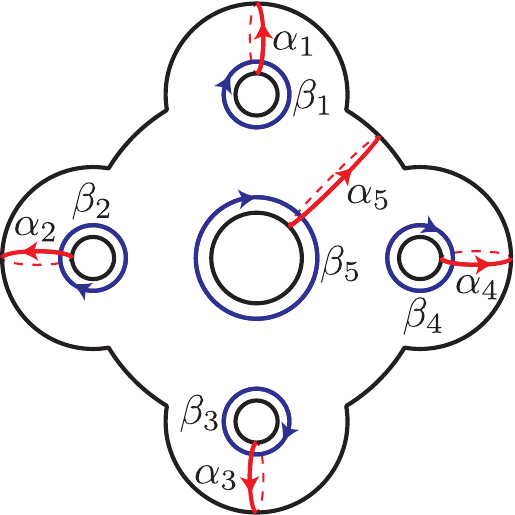}
	\caption{The generators for $H_1(\Sigma_5;\mathbb{Z})$.} 
	\label{fig:Genus5H1basis}
\end{figure}

Then, the mod $2$ homology classes of the vanishing cycles are
\begin{align*}
	d=d_1=d_2= \beta_5, \quad \tilde{d} = \alpha_1 + \alpha_2 +\alpha_3 + 
	\alpha_4 + \beta_5.
\end{align*}
We define the quadratic form $q$ on $H_1(\Sigma_5;\mathbb{Z}_2)$ by setting
\begin{align*}
	q(\alpha_i)=q(\beta_j)=1 \text{ for all $i,j$ except for } q(\alpha_5)=0
\end{align*}
and we let $s$ to be the corresponding spin structure on $\Sigma_5$.
This quadratic form satisfies  $q(d)=q(d_1)=q(d_2)=q(\tilde{d})=1$.
Next, observe that $q(a_i)=q(b_i)=q(c_i)=1$ for all $i$ and $q(d_3)=1$, hence 
Dehn twists along those curves are in $\M(\Sigma_5,s)$.
It follows that the two entries of the commutator in the positive factorization~\eqref{eq:genus5LFovertorus_spin}, 
$t_{c_1}t_{c_2}t_{c_3}t_{c_4} t_{a_4}^{-1}t_{a_3}^{-1}t_{a_2}^{-1}t_{a_1}^{-1}$
and $t_{d_3}^{-1}\phi$,
preserve $q$ since each one only involves Dehn twists in $\M(\Sigma_5,s)$.
Therefore, the quadratic form $q$ prescribes a spin structure on $Y$.

\smallskip
\noindent
\underline{An infinite family of spin fibrations over $T^2$}:
Going back to the construction of the relation~\eqref{eq:genus5LFovertorus_v1}, let us consider the diffeomorphisms 
\begin{align*}
	\phi_n = t_{a_1}^n \phi
\end{align*}
for $n \in \mathbb{Z}$.
Since the curve $a_1$ is disjoint from $\sigma_i$, $a_i$, and $a_4^\prime$, the  mapping class $\phi_n$ also maps $(a_i,c_i)$ to $(\sigma_i,a_i)$ for $i=1,2,3$ and $(a_4,c_4)$ to $(\sigma_4,a_4^\prime)$.
Therefore, we can replace $\phi$ by $\phi_n$ in the construction of~\eqref{eq:genus5LFovertorus_spin} and thus obtain
\begin{align} \label{eq:genus5LFoverTorusInfinite}
	\DT{d_2}\DT{d}\DT{d_1}\DT{\tilde{d}} \;
	[\psi, \LDT{d_3}\phi_n]
	= \DT{\delta}^0,
\end{align} 
where $\psi=t_{c_1}t_{c_2}t_{c_3}t_{c_4} t_{a_4}^{-1}t_{a_3}^{-1}t_{a_2}^{-1}t_{a_1}^{-1}$. This prescribes a genus--$5$ Lefschetz fibration $h_n \colon Y_n \to T^2$ with a section $s'_n  \in Y_n$ of self-intersection zero, for each $n \in \Z$.  

Note  that $q(a_1)=1$ where $q$ is the quadratic form we described above. 
Hence, $t_{a_1}$ preserves $q$, and so does $t_{d_3}^{-1} \phi_n = t_{d_3}^{-1} t_{a_1}^n \phi$.
We have already seen that $\psi$ preserves $q$ and that $q(d)=q(d_1)=q(d_2)= q(\tilde{d})= 1$.
Therefore, the quadratic form $q$ still prescribes a spin structure on $Y_n$.

More generally,  using any embedding  $\Sigma_5^1 \hookrightarrow \Sigma_g^1$,  we can view the positive factorizations
~\eqref{eq:genus5LFoverTorusInfinite} in $\M(\Sigma_g^1)$ instead,  for each $g \geq 5$.  Each one of these factorizations in $\M(\Sigma_g^1)$ now prescribes a genus--$g$ Lefschetz fibration \mbox{$h_{g,n} \colon Y_{g,n} \to T^2$} with a section $s'_{g,n} \subset Y_{g,n}$ of self-intersection zero.  Note that   $(Y_n, h_n)=(Y_{5,n}, h_{5,n})$.  Furthermore,  if we  take the restriction of the spin structure $s$ on $\Sigma_5$ to the subsurface $\Sigma_5^1$,  and through the above embedding,  extend it with \emph{any} spin structure on $\Sigma_{g-5}^1$,  we will get a spin structure on $\Sigma_g= \Sigma_5^1 \cup \Sigma_{g-5}^1$ that is also preserved by the monodromy of the genus--$g$ fibration,  for each $g \geq 5$. 

For each fixed $g \geq 5$,  the total spaces in the family $\{ Y_{g,n} \, | \, n \in \N \}$ will be pairwise homotopy inequivalent.  Once again, this will follow from having distinct first homology groups, which we will show next.

\smallskip
\noindent \underline{First homology calculation}: Next we will compute $H_1(Y_n)$ using the above geometric basis $\mathcal{B}'$ for $H_1(\Sigma_5; \Z)$. 
The homology group $H_1(Y_n)$ is isomorphic to the quotient of $H_1(T^2) \oplus H_1(\Sigma_5)$ modulo the following types of relations:
\begin{enumerate}[(i)]
	\item
	$\eta=0$ for each  representative $\eta$ of the Lefschetz vanishing cycles $d, d_1, d_2$ and $d_3$ (taken with auxiliary orientations) expressed in  the basis $\mathcal{B}'$,  and
	\item
	$\psi(\gamma) = \gamma$ and 
	$\phi_n(\gamma) =\gamma$ for each $\gamma \in \mathcal{B}'$.
\end{enumerate}

We compute and list the necessary data below.
\begin{align*}
	&d = \beta_5; &
	&d_1 = -2\alpha_5 +\beta_5; \\
	&d_2 = 2\alpha_5 +\beta_5; &
	&d_3 = 4\alpha_5 +\beta_5; \\
	&\psi(\alpha_1) = \alpha_1; & 
	&\psi(\beta_1) = \alpha_5 +\beta_1; \\
	&\psi(\alpha_2) = \alpha_2; &
	&\psi(\beta_2) = \alpha_5 +\beta_2; \\
	&\psi(\alpha_3) = \alpha_3; &
	&\psi(\beta_3) = \alpha_5 +\beta_3; \\
	&\psi(\alpha_4) = \alpha_4; &
	&\psi(\beta_4) = \alpha_5 +\beta_4; \\
	&\psi(\alpha_5) = \alpha_5; &
	&\psi(\beta_5) = \alpha_1 +\alpha_2 +\alpha_3 +\alpha_4 -4\alpha_5 +\beta_5; \\
	&\phi_n(\alpha_1) = \alpha_1 +\alpha_5; &
	&\phi_n(\beta_1) = -n\alpha_1 -\alpha_5 +\beta_1; \\
	&\phi_n(\alpha_2) = \alpha_2 +\alpha_5; &
	&\phi_n(\beta_2) = -\alpha_5 +\beta_2; \\
	&\phi_n(\alpha_3) = \alpha_3 +\alpha_5; &
	&\phi_n(\beta_3) = -\alpha_5 +\beta_3; \\
	&\phi_n(\alpha_4) = \alpha_4 +\alpha_5; &
	&\phi_n(\beta_4) = -\alpha_5 +\beta_4; \\
	&\phi_n(\alpha_5) = \alpha_5; &
	&\phi_n(\beta_5) = (n-1)\alpha_1 -\alpha_2 -\alpha_3 -\alpha_4 -4\alpha_5 -\beta_1 -\beta_2 -\beta_3 -\beta_4 +\beta_5.
\end{align*}
Hence, we obtain the following nontrivial relations:
\begin{align}
	&\beta_5 =0; \label{rel:genus5LF1}\\ 
	&-2\alpha_5 +\beta_5=0; \label{rel:genus5LF2}\\
	&2\alpha_5 +\beta_5 =0; \label{rel:genus5LF3}\\
	&4\alpha_5 +\beta_5 =0; \label{rel:genus5LF4}\\
	&\alpha_5 +\beta_1 =\beta_1; \label{rel:genus5LF5}\\ 
	&\alpha_5 +\beta_2 =\beta_2; \label{rel:genus5LF6}\\
	&\alpha_5 +\beta_3 =\beta_3; \label{rel:genus5LF7}\\
	&\alpha_5 +\beta_4 =\beta_4; \label{rel:genus5LF8}\\
	&\alpha_1 +\alpha_2 +\alpha_3 +\alpha_4 -4\alpha_5 +\beta_5 =\beta_5; \label{rel:genus5LF9}\\
	&\alpha_1 +\alpha_5 =\alpha_1; \label{rel:genus5LF10}\\
	&\alpha_2 +\alpha_5 =\alpha_2; \label{rel:genus5LF11}\\
	&\alpha_3 +\alpha_5 =\alpha_3; \label{rel:genus5LF12}\\
	&\alpha_4 +\alpha_5 =\alpha_4; \label{rel:genus5LF13}\\
	&-n\alpha_1 -\alpha_5 +\beta_1 =\beta_1; \label{rel:genus5LF14}\\
	&-\alpha_5 +\beta_2 =\beta_2; \label{rel:genus5LF15}\\
	&-\alpha_5 +\beta_3 =\beta_3; \label{rel:genus5LF16}\\
	&-\alpha_5 +\beta_4 =\beta_4; \label{rel:genus5LF17}\\
	&(n-1)\alpha_1 -\alpha_2 -\alpha_3 -\alpha_4 -4\alpha_5 -\beta_1 -\beta_2 -\beta_3 -\beta_4 +\beta_5 =\beta_5. \label{rel:genus5LF18}
\end{align}
We have $\beta_5=0$ from~\eqref{rel:genus5LF1} and $\alpha_5=0$ from~\eqref{rel:genus5LF5}.
With $\alpha_5=0$ the relation~\eqref{rel:genus5LF14} reduces to $n\alpha_1=0$.
From~\eqref{rel:genus5LF9} with $\alpha_5=\beta_5=0$ we get $\alpha_4 = -\alpha_1 -\alpha_2 -\alpha_3$.
Then we substitute $n\alpha_1=\alpha_5=\beta_5=0$ and $\alpha_4 = -\alpha_1 -\alpha_2 -\alpha_3$ in~\eqref{rel:genus5LF18} to obtain $\beta_4 = -\beta_1 -\beta_2 -\beta_3$.
The reduced relations so far are
\begin{align*}
	\alpha_4 &= -\alpha_1 -\alpha_2 -\alpha_3; \\
	\beta_4 &= -\beta_1 -\beta_2 -\beta_3; \\
	\alpha_5 &=0; \\
	\beta_5 &= 0; \\
	n\alpha_1 &=0.
\end{align*}
We see that the above five relations generate all  the relations~\eqref{rel:genus5LF1}--\eqref{rel:genus5LF18}.
Therefore,
\begin{align}
	H_1(Y_n)=H_1(T^2) \oplus \mathbb{Z}\left< \alpha_1, \alpha_2, \alpha_3, \beta_1, \beta_2, \beta_3 \;|\; n\alpha_1=0 \right> = \mathbb{Z}^7 \oplus \mathbb{Z}/n\mathbb{Z}.
\end{align}

Lastly, since the monodromy action of $\pi_1(T^2)$ on $H_1(\Sigma_g)$ is a trivial extension of its action on $H_1(\Sigma_5^1)$,  the  homology calculation above also implies that $H_1(Y_{g,n}) \cong \mathbb{Z}^{2g-3} \oplus \mathbb{Z}/n\mathbb{Z}$.

\medskip
\noindent
\underline{Algebraic topology of $Y_{g,n}$}: As before, we can calculate the topological invariants of $Y_{g,n}$ using the explicit positive factorization for $h_{g,n}$.  The Euler characteristic of $Y_{g,n}$ is  given by 
\[\eu(Y_{g,n})=\eu(\Sigma_g) \, \eu(T^2) + \ell= 0+ 4 = 4, \]
where $\ell$ is the number of Dehn twists in the factorization.  

The signature of $Y_{g,n}$ is given by an algebraic count of the mapping class group relations we have employed to derive the  positive factorization for the fibration $h_{g,n}$ in  $\M(\Sigma_g)$.  We  used a single $2$--chain and seven lantern relations,  whereas all the other steps,  including our variation  $\phi_n$ of $\phi$,  have no effect on the signature calculation.  So,  $\sigma(Y_{g,n})=0$ for all $g \geq 5$ and $n \in \Z$. 

The remaining homology groups of $Y_{g,n}$,  as well as $b_2^+$ and $b_2^-$ are then determined by their Euler characteristic,  signature and first homology as before.  In particular,  we get $H_2(Y_{g,n})= \Z^{4g-4} \oplus \,\Z / n \Z$ and $b_2^{\pm}(Y_{g,n})= 2g-2$ for $n \neq 0$,  whereas $H_2(Y_{g,0})= \Z^{4g-2}$ and $b_2^{\pm}(Y_{g,0})= 2g-1$.

As we argued above,  each $(Y_{g,n}, h_{g,n})$ has  spin monodromy, and, moreover,  admits a section $s'_{g,n}$ of even self-intersection.  Thus,  by Proposition~\ref{prop:spin} every $Y_{g,n}$ is spin.

\medskip
\noindent
\underline{Differential topology of $Y_{g,n}$}: The fibration  $(Y_{g,n}, h_{g,n})$ can be equipped with a Gompf--Thurston symplectic form so that the fibers and the section $s'_{g,n}$ are  symplectic.  Since $Y_{g,n}$ is also spin, it is a minimal symplectic $4$--manifold.  We calculate
\[  c_1^2(Y_{g,n})= 8 \ \text{ and } \ \chi_h(Y_{g,n})=1 \]
for any $g \geq 5$ and $n \in \Z$.  The symplectic Kodaira dimension of $Y_{g,n}$ is $\kappa(Y_{g,n})=2$.

For each genus $g \geq 5$,  the total spaces in the family $\{ Y_{g,n} \, | \, n \in \N \}$ have distinct first homology groups, and therefore they are homotopy inequivalent.  As before,  if needed,   passing to an infinite countable subset and  relabeling the index set as $\N$ again,  we can get an infinite family of spin symplectic $4$--manifolds $\{Y_{g,n} \,  | \, n \in \N \}$,  none of which  are  homotopy equivalent to any compact complex surface.\footnote{Here for $n\neq 0$ the first Betti number of  $ Y_{g,n}$ is odd,  so  it cannot be a K\"{a}hler surface,  and in turn (since $\kappa(Y_{g,n})=2$),  it already cannot be a complex surface for any $g \geq 5$ and $n \neq 0$. }
\end{proof}

\begin{remark} \label{rk:smaller}
Other examples of Lefschetz fibrations over the $2$--torus with very few nodes were given in  \cite{HamadaMinimal, StipsiczYun, CataneseEtal}.  A few words on why the examples we obtained above are (almost) as good as it gets: Assume that $(X,f)$ is a genus--$g$ Lefschetz fibration over the $2$--torus with $\ell$ nodes and $\sigma(X)=0$. The existence of an almost complex structure on $X$ implies that $\eu(X)+ \sigma (X) \equiv 0$ (mod $4$), and since $\eu(X)=\ell$, we conclude that $\ell=4$ is the smallest number of nodes one can have in this case. As for the fiber genus, note that by the signature formula for hyperelliptic Lefschetz fibrations \cite{Matsumoto, Endo}, one cannot expect to have $\sigma(X)=0$ when $g=1$ or $2$. Thus, with our genus--$4$ non-spin examples in Theorem~\ref{thm:nonspinLFs} in mind (cf. \cite{CataneseEtal} for a spin example), the only remaining possibility for an even smaller example is a genus--$3$ fibration, which we have not been able to produce. 

We should also note that for our applications to follow, along with some conditions on the configuration of the vanishing cycles, the existence of a self-intersection zero section is  essential. The genus--$5$ examples in Theorem~\ref{thm:spinLFs} were the best we could obtain with this additional property in the spin case.
\end{remark}

\smallskip
\section{Exotic smooth structures on simply connected signature zero $4$--manifolds} 

We build the promised families of irreducible $4$--manifolds with signature zero in this section; first the non-spin families, and then the spin ones.  

To produce symplectic $4$--manifolds homeomorphic but not diffeomorphic to connected sums of $\CP \# \CPb$ and $S^2 \times S^2$, we will first extend the Lefschetz fibrations over the $2$--torus given in Theorems~\ref{thm:nonspinLFs} and~\ref{thm:spinLFs}, respectively, to a  Lefschetz fibration over a genus--$2$ surface.  These \emph{model} symplectic $4$--manifolds will have intersection forms that are stabilizations of the intersection forms of 
$\#_{2m+1} (\CP \,\#\, \CPb)$  and 
$\#_{2n+1} (S^2 \x S^2)$ with many copies of the hyperbolic matrix
$H=\begin{pmatrix}
0 & 1 \\
1 & 0 
\end{pmatrix}$. \linebreak
They will have a larger fundamental group  (in regards to the minimal number of generators) than that of the subfibration we began with, but will now contain many homologically essential Lagrangian tori that  carry the $\pi_1$ generators. Performing a Luttinger surgery along the link of Lagrangian tori in this Lefschetz fibration, we will be able to get a new symplectic $4$--manifold whose $\pi_1$ is normally generated by a pair of curves on the fiber. Two of the vanishing cycles we arranged for the subfibration (\emph{the kernel}) from Theorems~\ref{thm:nonspinLFs} and~\ref{thm:spinLFs} will kill these two curves, allowing us to conclude that we have a simply connected $4$--manifold. Then we vary this construction to produce the desired infinite families.  

Having given away the plot, let us now provide the details.

\subsection{Irreducible $4$--manifolds homeomorphic to $\#_{2m+1} (\CP \,\#\, \CPb)$} \

The following constitutes the non-spin half of Theorem~A:

\begin{theorem} \label{thm:nonspinexotic}
For each integer $m \geq 4$,  there exist an infinite family of pairwise non-diffeomorphic irreducible $4$--manifolds $\mathcal{F}_m:=\{ Z_{m,k} \, | \, k \in \Z^+ \}$ in the homeomorphism class of $\#_{2m+1} (\CP \,\#\, \CPb)$,  such that each $Z_{m,k}$ is obtained by a  surgery along the same link $\mathcal{L}_m$ of Lagrangian tori  in a fixed genus \mbox{$g=m+1$} symplectic Lefschetz fibration $\mathcal{Z}_m$ over $\Sigma_2$.  Each $\mathcal{F}_m$ contains a minimal symplectic \mbox{$4$--manifold} and infinitely many irreducible $4$--manifolds that do not admit a symplectic structure.  
\end{theorem}

\begin{proof}
By adding a trivial commutator to the monodromy factorization~\eqref{eq:genus4LFoverTorusMonodromyInfinite} (taken with $n=0$) of the genus $g \geq 5$ Lefschetz fibration $(X_{g, 0}, f_{g, 0})$ we described in the proof of Theorem~\ref{thm:nonspinLFs}, we get a monodromy factorization 
\begin{align}  \label{eq:modelmonodromy1}
	\DT{u_1}\DT{v_1}\DT{u_2}\DT{v_2} 
 [\DT{a_1}\DT{a_2}\LDT{d_2}\LDT{d_1}, \phi_0 \,] [1, 1]
 = \DT{\delta}^0 \,   \text{ \ \ in } \M(\Sigma_g^1) \, 
\end{align} 
for a genus $g=m+1$ Lefschetz fibration $\mathcal{Z}_m$ over $\Sigma_2$ with a section of self-intersection zero.  All the mapping classes on the left-hand side of this monodromy factorization are supported in a subsurface $\Sigma_4^1 \subset \Sigma_g^1$, so the fibration $\mathcal{Z}_m \to \Sigma_2$ is a trivial extension of the complement of the section and a regular fiber of the genus--$4$ Lefschetz fibration $(X_{4,0}, f_{4,0})$, the kernel $\mathcal{K}$ of our model $\mathcal{Z}_m$.  In other words, we have a decomposition 
\[ 
\mathcal{Z}_m= (\mathcal{K} \,  \cup \,  (\Sigma_{g-5}^2 \x \Sigma_1^1))  \, \cup 
(\Sigma_1^1 \x \Sigma_2) \, \cup \,  (\Sigma_{g-1}^1 \x \Sigma_1^1) .
\]
where the fibration on $\mathcal{Z}_m \setminus  \mathcal{K}$ is just the projection onto the second factor for each  product of surfaces.  
We equip  $\mathcal{Z}_m$ with a Gompf-Thurston symplectic form, which restricts to product symplectic forms on the latter two pieces.
Label the three main  pieces in the above decomposition as 
$\mathcal{A},  \mathcal{B}_0$ and $\mathcal{C}_0$,  in the same order.  See~Figure~\ref{fig:schematic} for a schematic.  
\begin{figure}[htbp]
	\centering
	\includegraphics[width=380pt]{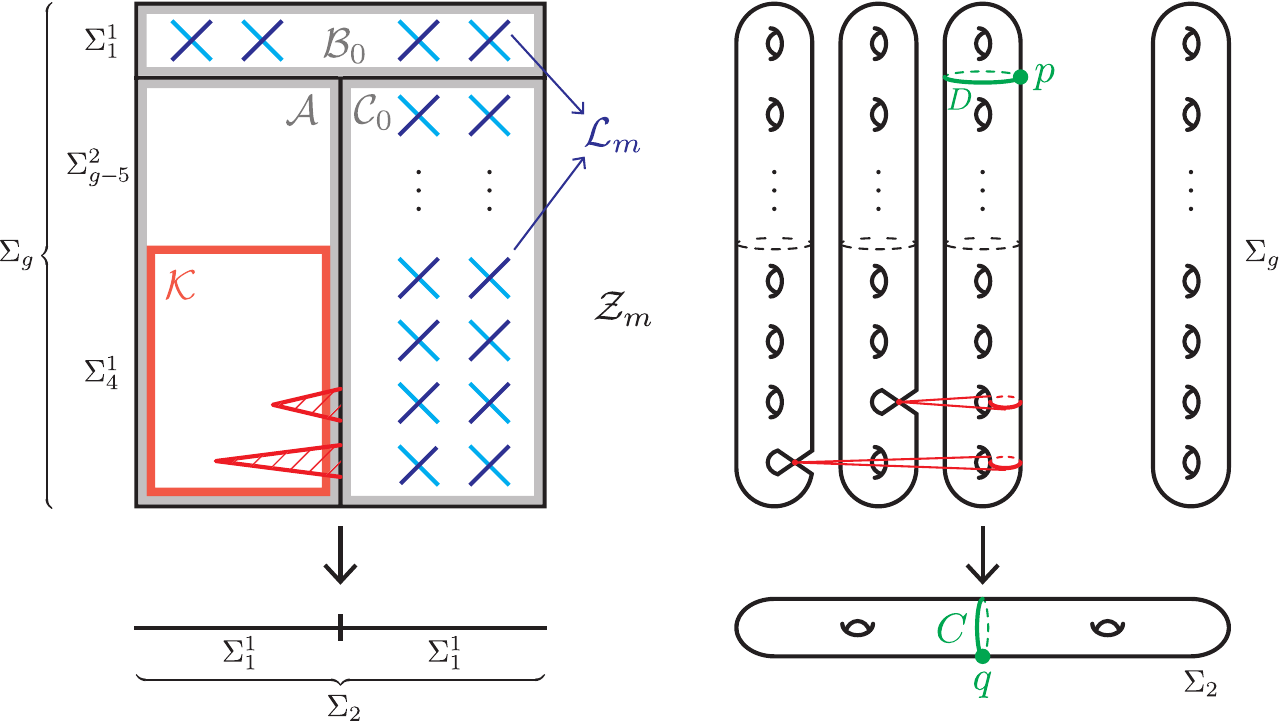}
	\caption{The decomposition of the Lefschetz fibration $\mathcal{Z}_m = \mathcal{A} \cup \mathcal{B}_0 \cup \mathcal{C}_0$. The dark blue represents the components of the link $\mathcal{L}_m$, whereas the light blue is the geometrically dual Lagrangian tori. The kernel contains the two Lefschetz thimbles bounding two different curves on the fiber.
 } 
	\label{fig:schematic}
\end{figure}

Consider the vanishing cycles  $u_2$ and $v_2$ in the monodromy of $\mathcal{Z}_m$,  which can be viewed on $\Sigma_4^1$ fiber of $\mathcal{K}$. Let $\psi = t_7^{-1} t_6^{-1} t_5^{-1} t_4^{-1} t_3^{-1} t_2 t_1$ where the Dehn twist curves are as indicated in Figure~\ref{fig:Genus4psi}. This map sends $u_2$ and $v_2$ to the pair of curves on $\Sigma_4^1$ shown in Figure~\ref{fig:Genus4x2y2}.  
We compute the representatives of $\psi(u_2)$ and $\psi(v_2)$ in $\pi_1(\Sigma_4^1)$ with respective to the standard generators given in Figure~\ref{fig:Genus4pi1generators} as
\begin{align*}
	\psi(u_2)&= a_2,\\
	\psi(v_2)&= a_1 a_2 b_2^{-1} a_3 [b_3,a_3] b_2.
\end{align*}

\begin{figure}[htbp]
	\centering
	\subfigure[Dehn twist curves for $\psi$. \label{fig:Genus4psi}]
	{\includegraphics[width=150pt]{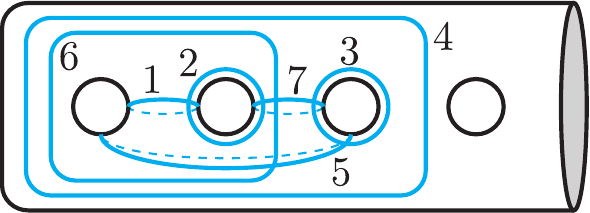}} 
	\hspace{15pt}
	\subfigure[The generators for $\pi_1(\Sigma_4^1)$. \label{fig:Genus4pi1generators}]
	{\includegraphics[width=150pt]{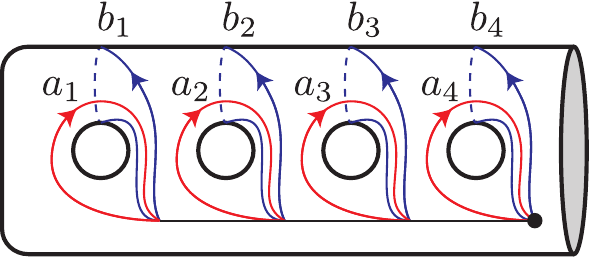}} 
	\hspace{10pt}
	\subfigure[The images of $u_2$ and $v_2$ under $\psi$. \label{fig:Genus4x2y2}]
	{\includegraphics[width=320pt]{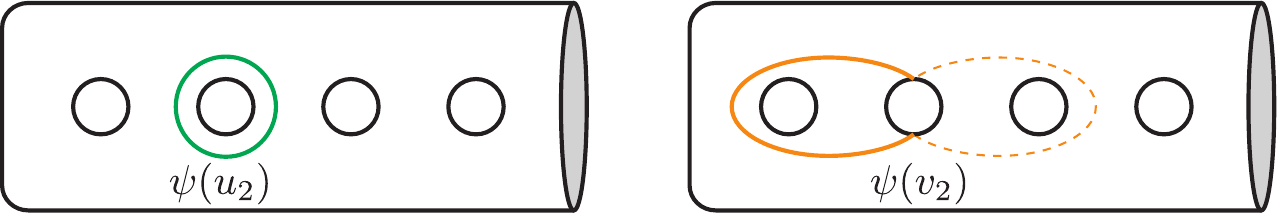}}  
	\caption{Topological configuration of $u_2$ and $v_2$. }
	\label{fig:Genus4pi1condition}
\end{figure}

We can now explain how to pick the link of Lagrangian tori $\mathcal{L}_m$ in $\mathcal{Z}_m$ so that a Luttinger surgery along  $\mathcal{L}_m$ results in a simply connected $4$--manifold.   Take a presentation for $\pi_1(\mathcal{Z}_m)$ as in Proposition~\ref{prop:pi1}; in particular,  it can be generated by the standard  bases $\{a_i, b_i\}$ for $\pi_1(\Sigma_g)$ and $\{x_j, y_j\}$ for $\pi_1(\Sigma_2)$.  We assume the monodromies over $\{x_1, y_1\}$ yield the non-trivial commutator in~\eqref{eq:modelmonodromy1} and over $\{x_2,y_2\}$ yield the trivial commutator.  Here we can identify $\Sigma_g$ and $\Sigma_2$  with a fiber over a point $q \in C$ in the base $\Sigma_2=\Sigma_1^1 \cup_C \Sigma_1^1$ and a section of self-intersection zero $\{ p \} \x \Sigma_2$ with $p \in D$ for $\Sigma_g = \Sigma_{g-1}^1 \cup_D \Sigma_1^1$, respectively,  and take $p \x q$ to be the base point (see Figure~\ref{fig:schematic}).  

By Proposition~\ref{prop:Luttinger}(a),  there is a Luttinger surgery along a link of Lagrangian tori in the second piece,  $\mathcal{B}_0=\Sigma_1^1 \x \Sigma_2$ (through the symplectomorphism switching the factors of the Euclidean product),  resulting in a symplectic $4$--manifold $\mathcal{B}$ with $\pi_1(\mathcal{B})$ normally generated by two geometrically dual curves on $\{ pt \} \x \Sigma_2 \subset \partial \mathcal{B} =\partial \mathcal{B}_0$ after we connect the base point to $p \x q$ on the  boundary.  If needed,  taking a symplectomorphism of $\mathcal{B}_0$ in the very beginning,  we can assume that these two normally generating curves are the basis curves $\{x_2, y_2\}$ on $(p \x \Sigma_2) \cap \partial \mathcal{C}_0$.

Similarly, by Proposition~\ref{prop:Luttinger}(b),  for a pair of disjoint and homologically independent curves $\{a ,b\}$ on $\Sigma_{g-1}^1 \x q$, there is a  Luttinger surgery along a link of Lagrangian tori in $\mathcal{C}_0=\Sigma_{g-1}^1 \x \Sigma_1^1$,  resulting in a symplectic $4$--manifold $\mathcal{C}$ such that  the quotient of $\pi_1(\mathcal{C})$ by the normal closure of $b$ is a cyclic group generated by the image of $a$.  Let $a$ and $b$ be the curves in the free isotopy classes of $a_1 a_3$ and $a_2$, respectively.  
By taking a symplectomorphism of $\mathcal{C}_0$ (with $\widehat{\psi}^{-1} \x \rm{id}$, where $\widehat{\psi}$ is an extension of $\psi$ on $\Sigma_4^1 \subset \Sigma_g^1$ by identity) in the beginning, we can assume  that the normally generating pair is $\{a':=\psi^{-1}(a), b':=\psi^{-1}(b)\}$  instead of  $\{a, b\}$. 

Now,  consider the simultaneous Luttinger surgery along the link $\mathcal{L}_m$ of Lagrangian tori above in $\mathcal{Z}$,  resulting in a new symplectic $4$--manifold $Z_m:=\mathcal{A} \cup \mathcal{B} \cup \mathcal{C}$.  By Seifert–Van Kampen,  we conclude that $\pi_1(Z_m)$ is normally generated by $a'$ and $b'$ such that $\pi_1(Z_m) \, / \, N(b')$ is a cyclic group generated by $a'$. However,  as our calculation above shows, the Lefschetz thimble along $u_2$ kills $b'$,  and $v_2$,  which bounds another Lefschetz thimble,  now maps to $a'$ under the quotient map $\pi_1(Z_m) \to \pi_1(Z_m) \, / \, N(b')$.  Hence $\pi_1(Z_m) =1$ as promised. 

The Euler characteristic of our model is $\eu(\mathcal{Z}_m)=4(g-1)(2-1)+4=4g=4m+4$.  Since the monodromy factorization~\eqref{eq:modelmonodromy1}  is derived from~  
\eqref{eq:genus4LFoverTorusMonodromyInfinite} without using any non-trivial relators, we have the signature $\sigma(\mathcal{Z}_m)=\sigma(X_{g,0})=0$.  The reducible genus--$2$ fiber component corresponding to the separating vanishing cycle $v_1$ in $\mathcal{K}$ is also contained in $\mathcal{Z}_m$; the presence of this surface of self-intersection $-1$ indicates that the intersection form of $\mathcal{Z}_m$ is odd.\footnote{Note also that, since the monodromy action of the base is trivial on $\mathcal{Z}_m \setminus \mathcal{K}$,  we have $H_1(\mathcal{Z}_m)=\Z^{2m+1}$, thus the  first Betti number $b_1(\mathcal{Z}_m)$ is odd and our model $\mathcal{Z}_m$ is not K\"{a}hler. } Moreover,  since $\mathcal{Z}_m$ is a relatively minimal symplectic Lefschetz fibration over a positive genus surface,  it is minimal \cite{StipsiczMinimal}.

As we obtained $Z_m$ from $\mathcal{Z}_m$ by a Luttinger surgery along a link of tori $\mathcal{L}_m$ (away from $\mathcal{K}$), we have  $\eu(Z_m)=4m+4$,  $\sigma(Z_m)=0$ and $Q_{Z_m}$ is odd.  Since $\pi_1(Z_m)=1$,  we conclude that the intersection form of $Z_m$ is of rank $4m+2$.   Therefore,  $Q_{Z_m} \cong Q_{\#_{2m+1} (\CP \,\#\, \CPb)}$,   and by Freedman's celebrated theorem,  $Z_m$ is homeomorphic to $\#_{2m+1} (\CP \,\#\, \CPb)$.  Since there is a reverse Luttinger surgery from $Z_m$ to $\mathcal{Z}_m$ and any exceptional symplectic sphere can be assumed to be disjoint from a given collection of  Lagrangian tori,  the minimality of $\mathcal{Z}_m$ implies the minimality of the symplectic $4$--manifold $Z_m$,  and in turn, we conclude that it is irreducible.  

As in \cite{FPS, AkhmedovBaykurPark},  performing non-Luttinger surgery along one component of $\mathcal{L}_m$ (with integral surgery framing $k >1 $ instead of $1$) we get an infinite family $\{Z_{m,k}  \, | \, k \in \Z^+\}$ of irreducible $4$--manifolds that are still in the same homeomorphism class, but with distinct Seiberg-Witten invariants, where no $Z_{m,k}$, other than $Z_{m,1}=Z_m$, has any basic classes that evaluate to $\pm 1$; and thus, they do not admit symplectic structures. 
\end{proof}

\smallskip
\subsection{Irreducible $4$--manifolds homeomorphic to $\#_{2n+1} (S^2 \x S^2)$.} \

Next is the spin half of Theorem~A:

\begin{theorem} \label{thm:spinexotic}
For each integer $n \geq 5$,  there exist an infinite family of pairwise non-diffeomorphic irreducible $4$--manifolds $\mathcal{F}'_n:=\{ Z'_{n,k} \, | \, k \in \Z^+ \}$ in the homeomorphism class of $\#_{2n+1} (S^2 \x S^2)$,  such that each $Z'_{n,k}$ is obtained by a surgery along the same link $\mathcal{L}'_n$ of Lagrangian tori  in a fixed  genus \mbox{$g=n+1$} symplectic Lefschetz fibration $\mathcal{Z}'_n$ over $\Sigma_2$.  Each $\mathcal{F}'_n$ contains a minimal symplectic \mbox{$4$--manifold} and infinitely many irreducible $4$--manifolds that do not admit a symplectic structure. 
\end{theorem}

\begin{proof}
By adding a trivial commutator to the monodromy factorization~\eqref{eq:genus5LFoverTorusInfinite} of the genus $g \geq 6$ Lefschetz fibration $(Y_{g, 1}, h_{g, 1})$ we described in the proof of Theorem~\ref{thm:spinLFs}, we now get a monodromy factorization 
\begin{align}  \label{eq:modelmonodromy2} 
{\DT{d_2}\DT{d}\DT{d_1}\DT{\tilde{d}} \;
	[\psi, \LDT{d_3}\phi_1]} \, [1, 1]
 = \DT{\delta}^0 \,   \text{ \ \ in } \M(\Sigma_g^1) \, 
\end{align} 
for a genus $g=n+1$ spin Lefschetz fibration $\mathcal{Z}'_n$ over $\Sigma_2$ with a section of self-intersection zero.  All the mapping classes on the left-hand side  are supported in a subsurface $\Sigma_5^1 \subset \Sigma_g^1$, so the fibration $\mathcal{Z}'_n \to \Sigma_2$ is a trivial extension of the complement of the section and a regular fiber of the genus--$5$ Lefschetz fibration $(Y_{5,1}, h_{5,1})$, the kernel $\mathcal{K}'$ of our model $\mathcal{Z}'_m$.\footnote{By choosing the second parameter in this family of Lefschetz fibration as $1$, this time we get $H_1(\mathcal{Z}'_n ) = \Z^{2n+1}$, which once again means that our model is not a complex surface.} This time, we have a decomposition 
\[ 
\mathcal{Z}'_n= (\mathcal{K'} \,  \cup \,  (\Sigma_{g-6}^2 \x \Sigma_1^1))  \, \cup 
(\Sigma_1^1 \x \Sigma_2) \, \cup \,  (\Sigma_{g-1}^1 \x \Sigma_1^1) \, ,
\]
where the fibration on $\mathcal{Z}'_n \setminus  \mathcal{K}'$ is  the projection onto the second factor for each  product of surfaces.  
We equip  $\mathcal{Z}'_n$ with a Gompf-Thurston symplectic form.
Label the three main  pieces in the above decomposition as 
$\mathcal{A}',  \mathcal{B}'_0$ and $\mathcal{C}'_0$.

We can now proceed as in the proof of Theorem~\ref{thm:nonspinexotic}.  In fact we perform the same Luttinger surgeries in $\mathcal{B}'_0 = \mathcal{B}_0$ and in $\mathcal{C}'_0= \mathcal{C}_0$ (for $m=n$) to produce the symplectic $4$--manifold $Z'_n$.  Choosing the Luttinger surgery along the link $\mathcal{L}'_n$ of Lagrangian tori in $\mathcal{Z}'_n$ this way,  we generate $\pi_1(Z'_n)$ normally by two disjoint curves $a$ and $b$ that do not separate  $\Sigma_g \x q$,  which we can assume (by applying a symplectomorphism to the piece $\mathcal{C}'_0$ at the beginning) are the vanishing cycles $d$ and $\tilde{d}$.    As the  Lefschetz thimbles in $\mathcal{K}'$ are bounded by $a$ and $b$,  we get $\pi_1(Z'_n)=1$.  

By the same line of arguments and  calculations as in the proof of Theorem~\ref{thm:nonspinexotic},  we get  $\eu(Z'_n)=4n+4$ and $\sigma(Z'_n)=0$.  Since for each component of $\mathcal{L}'_n$, there is a dual Lagrangian torus,  together they constitute a distinct hyperbolic pair $H$ in $H_2(\mathcal{Z}'_n)$.  So the intersection form $Q_{\mathcal{Z}'_n} \cong Q_{Z'_n} \oplus r H$, where $r = \pi_0(\mathcal{L}'_n)$. In particular,  because $\mathcal{Z}'_n$ was spin, it had an even intersection form,  which implies that $Q_{Z'_n}$ is also even.  Since $H_1(Z'_n)$ has no $2$--torsion,  it follows that $Z'_n$ is spin.  We then derive the infinite family $\{Z'_{n,k}  \, | \, k \in \Z^+\}$ by performing a non-Luttinger surgery along a component with framing $k>1$ to produce the infinite family of irreducible spin $4$--manifolds in the same homeomorphism class, but not admitting a symplectic structure  except for $Z'_{n,1}=Z'_n$. 
\end{proof}

\smallskip
\subsection{Further results and remarks} \

We end with a few results that follow rather immediately from our constructions above  and some remarks on our examples.

\begin{add} \label{thm:addendum} \label{add:knotsurgery}
For each integer $m \geq 5$ and $n \geq 6$,  there exists an infinite family of pairwise non-diffeomorphic minimal symplectic $4$--manifolds, which are not diffeomorphic to any complex surfaces,  in the homeomorphism class of $\#_{2m+1} (\CP \, \# \, \CPb)$ and $\#_{2n+1} (S^2 \x S^2)$,   respectively.
\end{add}

\begin{proof}
Let's take a closer look at our proofs of Theorem~\ref{thm:nonspinexotic} and~\ref{thm:spinexotic}, starting with the former.  When $m \geq 5$,  we have (at least) one more pair of geometrically dual Lagrangian tori in the product $\Sigma_{g-5}^2 \x \Sigma_1^1$  contained within the first piece $\mathcal{A}=\mathcal{K} \cup (\Sigma_{g-5}^2 \x \Sigma_1^1)$ of our model $\mathcal{Z}_m$.  These descend to a pair of homologically essential Lagrangian tori in $Z_m$.  Let  $T$ be one of them.  As shown by Gompf \cite{Gompf}, we can perturb the symplectic form on $Z_m$ so that $T$ becomes a symplectic torus of self-intersection zero.  Moreover,  $\pi_1(Z_m  \setminus T) = \pi_1(Z_m)=1$,  which follows from the meridian of $T$ in $Z_m$ being a commutator of basis elements from $\pi_1(\Sigma_g)$ and $\pi_1(\Sigma_2)$ which are trivial in $\pi_1(Z_m \setminus \nu T)$. We can thus perform Fintushel and Stern's knot surgery \cite{FSKnotSurgery} along $T$ in $Z_m$,  using an infinite family of fibered knots with distinct Alexander polynomials, which results in an infinite family of symplectic $4$--manifolds that are pairwise  non-diffeomorphic but are all in the same homeomorphism class as $Z_m$. The same goes for the spin case when $n \geq 6$ by the same array of arguments.  All these minimal symplectic $4$--manifolds have $c_1^2>0$ and have many basic classes by the knot surgery formula for the Seiberg-Witten invariants, whereas any minimal complex surface of general type has only one basic class up to sign. We conclude that, provided $m \geq 5$ and $n \geq 6$, we can also produce infinitely many (minimal) symplectic $4$--manifolds that are in the desired homeomorphism classes and are not diffeomorphic to each other or to any complex surfaces.
\end{proof}

\begin{add} \label{thm:addendum2} \label{add:smallest}
For each integer $m \geq 4$ and $n \geq 5$,  there exists an infinite family of pairwise non-diffeomorphic $4$--manifolds, each including a symplectic member, in the homeomorphism class of $(T^2 \x S^2) \#_ {2m} (\CP \, \# \, \CPb)$ and $(T^2 \x S^2) \#_{2n} (S^2 \x S^2)$, respectively.
\end{add}

\begin{proof}
We first describe in detail the construction of the symplectic examples in the homeomorphism classes of $(T^2 \x S^2) \#_ {8} (\CP \, \# \, \CPb)$ and $(T^2 \x S^2) \#_{10} (S^2 \x S^2)$. These smallest signature zero exotic $4$--manifolds, with $b_2$ two less than that of the smallest simply connected examples we have produced in this paper, are the main reason for us to include this addendum. 

For the non-spin examples, let $\mathcal{\widehat{Z}}$ be a genus--$4$ Lefschetz fibration over $\Sigma_2$ with a section $\Sigma$ of self-intersection zero, given by the monodromy factorization
\begin{align}  \label{eq:pi1ZxZmonodromy}
	\DT{u_1}\DT{v_1}\DT{u_2}\DT{v_2} 
 [\DT{a_1}\DT{a_2}\LDT{d_2}\LDT{d_1}, \phi_0 \,] [1, 1]
 = \DT{\delta}^0 \,   \text{ \ \ in } \M(\Sigma_4^1) \, ,
\end{align} 
which decomposes as 
\begin{equation} \label{newdecomposition}
\mathcal{\widehat{Z}}= \mathcal{K} \, \cup 
(D^2 \x \Sigma_2) \, \cup \,  (\Sigma_{4}^1 \x \Sigma_1^1) \, ,
\end{equation}
where the third piece is the same as $\mathcal{C}_0$ we had in the decomposition of $\mathcal{Z}_m$ earlier. Let $\mathcal{\widehat{L}}$ be the same link of $8$ Lagrangian tori in $\mathcal{C}_0$. Performing the same Luttinger surgeries along these tori as we did in the previous section yields $\mathcal{C}$, and the resulting symplectic four-manifold $\widehat{Z}$ decomposes as $\widehat{Z}=\mathcal{K} \cup \nu \Sigma \cup \mathcal{C}$. 

Once again, let $\{a_i, b_i\}$ and $\{x_j, y_j\}$ be the standard generators of $\pi_1(F)$ and $\pi_1(\Sigma)$, where $F \cong \Sigma_4$ descends from the regular fiber of $\mathcal{\widehat{Z}}$ and $\Sigma \cong \Sigma_2$ from the $0$--section. Through the inclusion of $F \cup \Sigma$, these generate $\pi_1(\mathcal{\widehat{Z}})$, and in turn, $\pi_1(\widehat{Z})$. As in the proof of Theorem~\ref{thm:nonspinexotic}, after the Luttinger surgery, the generators $a_1, b_1, \ldots, a_4, b_4, x_2, y_2$ become normally generated by two curves $a, b$ on $F$, which get killed in $\pi_1(\widehat{Z})$ by the vanishing cycles of $\mathcal{K}$.\footnote{Unlike in the proof of Theorem~\ref{thm:nonspinexotic}, without any further Luttinger surgeries that were previously performed in the additional piece $\mathcal{B}_0$, $x_1$ and $y_1$ are no longer generated by these.}
The surface relator $[x_1,y_1] [x_2,y_2] = 1$ becomes $[x_1,y_1]=1$, so $\pi_1(\widehat{Z})$, which is a quotient of $\langle x_1, y_1 \, | \, [x_1, y_1] \rangle= \Z^2$, is abelian. The only other non-trivial relators involving $x_1$ or $y_1$ read as
\begin{align*}
    x_1^{-1} c \, x_1 A(c) &= 1 \, , \\
    y_1^{-1} c \, y_1 B(c) &= 1 \, , 
\end{align*}
for each $c \in \{a_i, b_i\}$ (where $A=\DT{a_1}\DT{a_2}\LDT{d_2}\LDT{d_1}$ and $B=\phi_0$), which become $c A(c) = c B(c)=1$ (in fact, they become trivial, since all the generators $a_i, b_i$ are killed at this point). 

Hence, $\pi_1(\widehat{Z})=\Z^2$. As we calculate $\eu(\widehat{Z})=16$, $\sigma(\widehat{Z})=0$ and note that $H_2(\widehat{Z})$ contains an odd class, we conclude that $\widehat{Z}$ is homeomorphic to $(T^2 \x S^2) \#_ {8} (\CP \, \# \, \CPb)$ by \cite{HambletonKreckTeichner}. However, $\widehat{Z}$ is minimal symplectic, so they are not diffeomorphic. 

For the spin examples, let $\mathcal{\widehat{Z}}'$ be a genus--$5$ Lefschetz fibration over $\Sigma_2$ with a section $\Sigma$ of self-intersection zero, given by the monodromy factorization
\begin{align}  \label{eq:pi1ZxZmonodromy'}
{\DT{d_2}\DT{d}\DT{d_1}\DT{\tilde{d}} \;
	[\psi, \LDT{d_3}\phi_1]} \, [1, 1]
 = \DT{\delta}^0 \,   \text{ \ \ in } \M(\Sigma_5^1) \, 
\end{align} 
which decomposes as 
\begin{equation} \label{newdecomposition'}
\mathcal{\widehat{Z}}'= \mathcal{K}' \, \cup 
(D^2 \x \Sigma_2) \, \cup \,  (\Sigma_{5}^1 \x \Sigma_1^1) \, ,
\end{equation}
where the third piece is the same as $\mathcal{C}_0$ we had in the decomposition of $\mathcal{Z}'_m$ earlier. Performing the same Luttinger surgeries in $\mathcal{C}_0$, we conclude that the resulting symplectic four-manifold $\widehat{Z}'$ we get decomposes as $\widehat{Z}'=\mathcal{K}' \cup \nu \Sigma \cup \mathcal{C}$. 

A similar array of arguments as in the non-spin case then shows that $\pi_1(\widehat{Z}')=\Z^2$. We calculate $\eu(\widehat{Z}')=20$ and $\sigma(\widehat{Z})=0$. The Lefschetz fibration $\mathcal{\widehat{Z}}'$ is spin by Proposition~\ref{prop:spin} and by similar arguments as in the proof of Theorem~\ref{thm:spinexotic}, the $4$--manifold $\widehat{Z}'$ we derived by Luttinger surgeries is also spin. We conclude that $\widehat{Z}'$ is homeomorphic to $(T^2 \x S^2) \#_ {10} (S^2 \x S^2)$ by \cite{HambletonKreckTeichner}, but $\widehat{Z}'$ is minimal symplectic, so they cannot be diffeomorphic. 

The examples with larger topology are obtained by increasing the genus of the initial fibration over $\Sigma_2$ and performing Luttinger surgeries along more tori in the (larger) $\mathcal{C}_0$ piece. In each case, we obtain the infinite families by performing non-Luttinger surgeries along one of the Lagrangian tori as in the constructions given in the previous section.
\end{proof}

\medskip
\begin{remark} \label{rk:BF}
If we segregate the irreducibility condition, we can also obtain infinitely many exotic copies on connected sums of even numbers of connected sums of $\CP \# \CPb$ or $S^2 \times S^2$, namely on 
$\#_{22+4l} (\CP \# \CPb)$ and $\#_{22+4l} (S^2 \times S^2)$, for any $l \in \N$.  For this, take for example the non-spin-family $\{ Z_{5,0} \# Z_{5+2l,k} \, | \, k \in \N\}$ and the spin family $\{ Z'_{5,0} \# Z'_{5+2l,k} \, | \, k \in \N\}$, and appeal to the non-vanishing of Bauer-Furuta stable cohomotopy Seiberg-Witten invariants \cite{Bauer}.  None of these $4$--manifolds admit symplectic structures, of course.
\end{remark}

\begin{remark}\label{rk:arbitrarypi1}
Through a variation of a celebrated result (and its proof) by Gompf, we also obtain the following result: Any finitely presented group is the fundamental group of a minimal symplectic $4$--manifold with $c_1^2 = 2 c_2$, which can be taken to be non-spin or spin. As noted in the proof of the Addendum~\ref{thm:addendum},  after perturbing the symplectic form,  we can find a self-intersection zero symplectic torus $T$ with a trivial $\pi_1$ complement  in  our exotic minimal symplectic $Z_5$, as well as in $Z'_6$.  Therefore,  we can take either one of these (instead of the elliptic surfaces) as the building block $W$ in Gompf's construction \cite[Proof of Theorem~4.1]{Gompf} to derive the non-spin and spin examples,  respectively.  Note that the genus--$2$ surface of odd self-intersection in $Z_5$ is away from $T$, so the $4$--manifolds we get after fiber summing along $T$ are still non-spin.  The minimality of the resulting symplectic $4$--manifolds after fiber sums follows from  \cite{Usher}. 
\end{remark}

\begin{remark}\label{rk:kirby}
While we do not care for it here, let us note that one can, if they desire to do so, draw a Kirby diagram for any one of our exotic $4$--manifolds in Theorem~A using the explicit descriptions we have provided for the Lefschetz fibrations $\mathcal{Z}_m$ and $\mathcal{Z}'_n$  and the Luttinger surgeries along $\mathcal{L}_m$ and $\mathcal{L}'_n$ in them, respectively.  After decomposing the  fibration as a surface bundle over $\Sigma_2^1$ and a Lefschetz fibration over $D^2$, following the explicit monodromy factorization we provide, one first draws a Kirby diagram of the surface bundle using Akbulut's algorithm \cite[Sections~4.1--4.3]{Akbulut}, which importantly allows one also to describe the link of tori we surger, and then attach the four Lefschetz $2$--handles along the regular fiber as usual.  Then, following the Luttinger surgery information we provided, one modifies the pieces of the diagram as prescribed in 
\cite[Sections~6.3--6.4]{Akbulut} to arrive at the final Kirby diagram.

With the exception of our earlier work in  \cite{BaykurHamada}, 
all other constructions of exotic simply connected symplectic $4$--manifolds with signature zero we know of in the literature involve fiber summing with a complex algebraic surface along an embedded holomorphic curve in it,  where the smooth topology (e.g. in terms of handle decompositions) of the latter key ingredient is only partially understood; and there appears to be no chance of describing Kirby diagrams of them.
\end{remark}

\begin{remark}\label{rk:kernel}  
Here is an epilogue intended for the avid reader.  It should not be hard to tell that the biggest challenge in our approach to building the desired irreducible $4$--manifolds is to have the right model symplectic $4$--manifold that we can kill the fundamental group of via surgeries along tori. Utilizing pairs of geometrically dual Lagrangian tori $T$ and $T'$ has been a well-exploited idea in the literature; a Luttinger surgery here allows one to express a $\pi_1$--generator carried on one of these two tori, say $T$, as its meridian, $\mu_T$, which is the boundary of the punctured $T'$. Provided all the $\pi_1$ generators of a model symplectic $4$--manifold $X_0$ are carried on a link $\mathcal{L}$ of Lagrangian tori (with a disjoint link of tori, geometrically dual to each component), a careful choice of surgeries may make it possible to express all the $\pi_1$ generators as commutators, resulting in a $4$-manifold $X$ with trivial $H_1(X)$. Nonetheless, to get trivial $\pi_1(X)$, one needs more than dual tori, namely, some collection of immersed disks relative to $X_0 \setminus \mathcal{L}$. Our repeated use of Proposition~\ref{prop:Luttinger} highlights how in favorable models, it suffices to have just two such disks. Or just one disk, if one is after producing an exotic $4$--manifold with cyclic fundamental group; see Remark~\ref{rk:Luttinger}. The \emph{kernels} are the most essential parts of the models we use, where we have these desired disks (in the form of Lefschetz thimbles in our examples) that bound curves which will normally generate the $\pi_1(X)$ after a clever choice of Luttinger surgeries in the model $X_0$. 

The challenge to produce small exotic $\#_{2m+1}(\CP \,\#\, \CPb)$ and $\#_{2n+1} (S^2 \times S^2)$ manifests itself in the algebraic setup:
it is much harder to come up with \emph{signature zero} positive factorizations for model symplectic Lefschetz fibrations.  In fact, the existence of non-trivial Lefschetz fibrations over the $2$--sphere with  signature zero was established only very recently \cite{BaykurHamada}.
\end{remark}

\medskip
\appendix
\section{Construction of the genus--$2$ pencil}
This appendix is devoted to providing a construction of the monodromy factorization~\eqref{eq:genus2MatsumotoRelationIIB} for the genus--$2$ Lefschetz pencil that we used in Section~\ref{Sec:fibrations}.  It is one of the factorizations obtained in an unpublished preprint \cite{Hamada} by the second author, which we single out and present the construction of here.

As a first step, we construct a variation of the \textit{six--holed torus relation}, which was first found by Korkmaz and Ozbagci \cite{KorkmazOzbagci}. 
Consider the six--holed torus in Figure~\ref{F:ConstructionOf6holedTorusRelation}.
\begin{figure}[htbp]
	\centering
	\subfigure[Four-holed torus bounded by $\{ \delta_3, \delta_4, z_1, z_2 \}$.
	\label{F:ConstructionOf6holedTorusRelation_1}]
	{\hspace{0pt}
		\includegraphics[height=120pt]{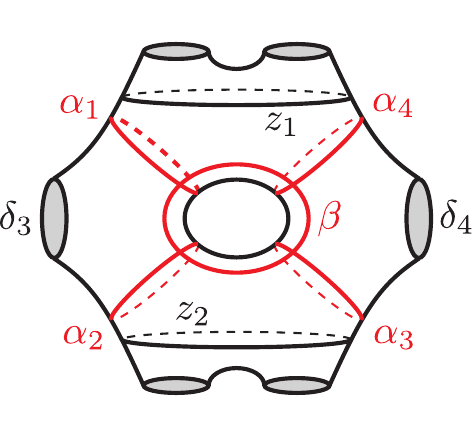}
		\hspace{0pt}} 
	\hspace{1em}	
	\subfigure[Two lantern configurations on the subsurfaces bounded by $\{ \delta_1, d_1, \alpha_1, \alpha_4 \}$ and $\{ \delta_2, d_2, \alpha_3, \alpha_2 \}$, respectively.
	\label{F:ConstructionOf6holedTorusRelation_2}]
	{\hspace{0pt}
		\includegraphics[height=120pt]{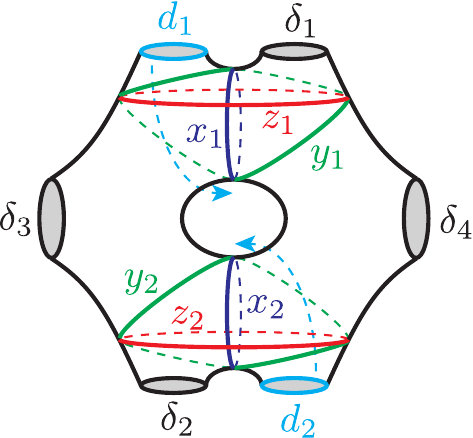}
		\hspace{0pt}} 
	\caption{Six-holed torus with boundary $\{ \delta_1, d_1, \delta_2, d_2, \delta_3, \delta_4 \}$.} \label{F:ConstructionOf6holedTorusRelation}	
\end{figure}
By the four--holed torus relation we have
\begin{align*} 
	\DT{\delta_3} \DT{\delta_4} \DT{z_1} \DT{z_2} &= (\DT{\alpha_1} \DT{\alpha_3} \DT{\beta} \DT{\alpha_2} \DT{\alpha_4} \DT{\beta})^2 \\		
	&= (\DT{\beta} \underline{\DT{\alpha_1} \DT{\alpha_3} \DT{\beta}} \DT{\alpha_2} \DT{\alpha_4} )^2 \\
	&= (\DT{\beta}  \DT{t_{\alpha_1} t_{\alpha_3}(\beta)} \DT{\alpha_1} \DT{\alpha_2} \DT{\alpha_3} \DT{\alpha_4} )^2
\end{align*}
in $\mathrm{Mod}(\Sigma_1^6)$.
Here, we used a cyclic permutation (since $t_{\beta}$ commutes with the left-hand side) and braid relations.
We rearrange the relation as follows:
\begin{align*} \allowdisplaybreaks
	\DT{\delta_3} \DT{\delta_4}   
	&= \DT{\beta}  \DT{t_{\alpha_1} t_{\alpha_3}(\beta)} \cdot \DT{\alpha_1} \DT{\alpha_4} \LDT{z_1} \cdot \DT{\alpha_3} \DT{\alpha_2} \cdot \DT{\beta}  \DT{t_{\alpha_1} t_{\alpha_3}(\beta)} \cdot \DT{\alpha_2} \DT{\alpha_3} \LDT{z_2} \cdot \DT{\alpha_1} \DT{\alpha_4}.
\end{align*}
Then consider the two lantern relations $t_{x_1} t_{y_1} t_{z_1} = t_{\alpha_1} t_{\alpha_4} t_{d_1} t_{\delta_1}$ and $t_{x_2} t_{y_2} t_{z_2} = t_{\alpha_2} t_{\alpha_3} t_{d_2} t_{\delta_2}$ with the curves as depicted in Figure~\ref{F:ConstructionOf6holedTorusRelation_2}, and substitute them to the above equation to obtain	
\begin{align*}
	\DT{\delta_3} \DT{\delta_4}   
	= \DT{\beta}  \DT{t_{\alpha_1} t_{\alpha_3}(\beta)} \cdot 
	\LDT{\delta_1} \LDT{d_1} \DT{x_1} \DT{y_1} \cdot 
	\DT{\alpha_3} \DT{\alpha_2} \cdot 
	\DT{\beta}  \DT{t_{\alpha_1} t_{\alpha_3}(\beta)} \cdot 
	\DT{x_2} \DT{y_2} \LDT{d_2} \LDT{\delta_2} \cdot 
	\DT{\alpha_1} \DT{\alpha_4}.
\end{align*}
This can be rearranged to
\begin{align*}  
	\DT{\delta_1} \DT{\delta_2} \DT{\delta_3} \DT{\delta_4} = 
	\DT{\beta} \DT{t_{\alpha_1} t_{\alpha_3}(\beta)} \cdot 
	\LDT{d_1} \DT{x_1} \DT{\alpha_3} \cdot \DT{y_1} \DT{\alpha_2} \cdot 
	\DT{\beta} \DT{t_{\alpha_1} t_{\alpha_3}(\beta)} \cdot 
	\DT{x_2} \DT{\alpha_1} \cdot \DT{y_2} \DT{\alpha_4} \LDT{d_2}. 
\end{align*}

The next step is to embed the six-holed torus into a genus--$2$ surface with four boundary components.
See Figures~\ref{F:ConstructionOf6holedTorusRelation_2} and~\ref{F:Construction_MatsumotoRelation_Lift}.
First, we push the boundary components $d_1$ and $d_2$ of the six-holed torus to the center along the dashed arrows.
Then, connect the two boundary components using an obvious tube.
In this way, we can regard that the last relation holds in $\mathrm{Mod}(\Sigma_2^4)$.
\begin{figure}[htbp]
	\centering
	\includegraphics[height=110pt]{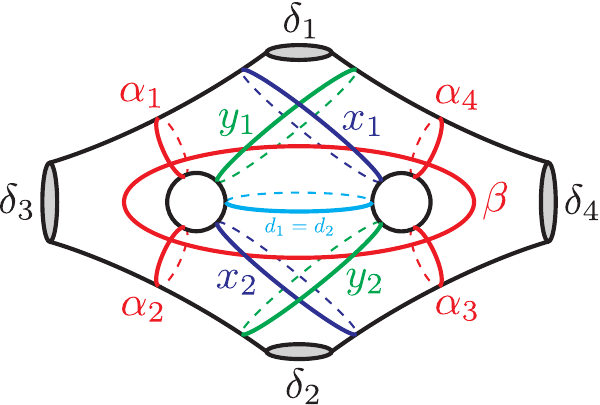}
	\caption{Four-holed surface of genus $2$ with boundary $\{ \delta_1, \delta_2, \delta_3, \delta_4 \}$, which is obtained by pushing $d_1$ and $d_2$ in Figure~\ref{F:ConstructionOf6holedTorusRelation_2} as indicated by the dashed arrows and then gluing them.} \label{F:Construction_MatsumotoRelation_Lift}	
\end{figure}
\begin{figure}[htbp]
	\centering
	\subfigure[Lantern configuration on the subsurface bounded by $\{ x_1, \alpha_3, x_2, \alpha_1 \}$.
	\label{F:Construction_MatsumotoRelation_Lift_L1}]
	{\includegraphics[height=90pt]{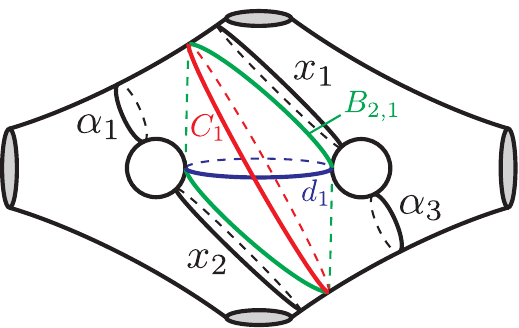}} 
	\hspace{.8em}	
	\subfigure[Lantern configuration on the subsurface bounded by $\{ y_1, \alpha_2, y_2, \alpha_4 \}$.
	\label{F:Construction_MatsumotoRelation_Lift_L2}]
	{\includegraphics[height=90pt]{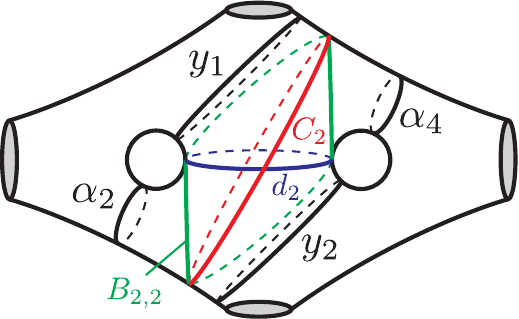}} 
	\caption{Lantern configurations.} \label{F:Construction_MatsumotoRelation_Lift_L1L2}	
\end{figure}

Using braid relations and cyclic permutations, we alter the equation,
\begin{align*}
	\DT{\delta_1} \DT{\delta_2} \DT{\delta_3} \DT{\delta_4} &= 
	\DT{\beta} \DT{t_{\alpha_1} t_{\alpha_3}(\beta)} \cdot 
	\LDT{d_1} \DT{x_1} \DT{\alpha_3} \cdot 
	\underline{\DT{y_1} \DT{\alpha_2} \cdot 
	\DT{\beta} \DT{t_{\alpha_1} t_{\alpha_3}(\beta)} \cdot 
	\DT{x_2} \DT{\alpha_1}} \cdot 
	\DT{y_2} \DT{\alpha_4} \LDT{d_2} \\
	&=
	\DT{\beta} \DT{t_{\alpha_1} t_{\alpha_3}(\beta)} \cdot 
	\LDT{d_1} \DT{x_1} \DT{\alpha_3} \DT{x_2} \DT{\alpha_1} \cdot
	\underline{\DT{y_1} \DT{\alpha_2} \cdot
	\DT{t_{\alpha_1}^{-1}t_{x_2}^{-1}(\beta)} \DT{t_{x_2}^{-1} t_{\alpha_3}(\beta)}} \cdot 
	\DT{y_2} \DT{\alpha_4} \LDT{d_2} \\
	&=
	\DT{\beta} \DT{t_{\alpha_1} t_{\alpha_3}(\beta)} \cdot 
	\LDT{d_1} \DT{x_1} \DT{\alpha_3} \DT{x_2} \DT{\alpha_1} \cdot
	\DT{t_{y_1} t_{\alpha_2}t_{\alpha_1}^{-1}t_{x_2}^{-1}(\beta)} \DT{t_{y_1} t_{\alpha_2}t_{x_2}^{-1} t_{\alpha_3}(\beta)} \cdot 
	\DT{y_1} \DT{\alpha_2} \DT{y_2} \DT{\alpha_4} \LDT{d_2}.
\end{align*}
Finally we find two lantern relations $t_{x_1} t_{\alpha_3} t_{x_2} t_{\alpha_1} = t_{B_{2,1}} t_{C_1} t_{d_1}$ and $t_{y_1} t_{\alpha_2} t_{y_2} t_{\alpha_4} = t_{B_{2,2}} t_{C_2} t_{d_2}$ as in Figure~\ref{F:Construction_MatsumotoRelation_Lift_L1L2}, which can be substituted to the equation and we get
\begin{align*}
	\DT{\beta} \DT{t_{\alpha_1} t_{\alpha_3}(\beta)} \cdot 
	\DT{B_{2,1}} \DT{C_1} \cdot 
	\DT{t_{y_1} t_{\alpha_2}t_{\alpha_1}^{-1}t_{x_2}^{-1}(\beta)} \DT{t_{y_1} t_{\alpha_2}t_{x_2}^{-1} t_{\alpha_3}(\beta)} \cdot
	\DT{B_{2,2}} \DT{C_2}
	=
	\DT{\delta_1} \DT{\delta_2} \DT{\delta_3} \DT{\delta_4}.
\end{align*}
One can directly check that $\beta = B_{0,1}$, $t_{\alpha_1} t_{\alpha_3}(\beta) = B_{1,1}$, $t_{y_1} t_{\alpha_2}t_{\alpha_1}^{-1}t_{x_2}^{-1}(\beta) = B_{0,2}$ and $t_{y_1} t_{\alpha_2}t_{x_2}^{-1} t_{\alpha_3}(\beta) = B_{1,2}$ (see Figure~\ref{fig:Genus2Matsumoto_LiftIIB}), hence the last equation is nothing but the claimed relation
\begin{align*}
	\DT{B_{0,1}} \DT{B_{1,1}} \DT{B_{2,1}} \DT{C_1} \DT{B_{0,2}} \DT{B_{1,2}} \DT{B_{2,2}} \DT{C_2}
	=\DT{\delta_1} \DT{\delta_2} \DT{\delta_3} \DT{\delta_4}.
\end{align*}

\end{document}